\documentclass[a4paper]{amsart}
\usepackage{amsmath,amssymb,times, amscd, graphicx, 
xspace,mathrsfs,a4wide,url,chngcntr,mathtools}
\usepackage[pagebackref=true]{hyperref}
\usepackage[title]{appendix}
\hypersetup{
  colorlinks   = true, 
  urlcolor     = blue, 
  linkcolor    = blue, 
  citecolor   = blue 
}

\newcommand{\Z}{\ensuremath{\mathbb{Z}}\xspace}
\newcommand{\Q}{\ensuremath{\mathbb{Q}}\xspace}

\newcommand{\C}{\ensuremath{\mathbb{C}}\xspace}
\newcommand{\A}{\ensuremath{\mathbb{A}}\xspace}

\newcommand{\Qp}{\ensuremath{\mathbb{Q}_{p}}\xspace}
\newcommand{\Zp}{\ensuremath{\mathbb{Z}_{p}}\xspace}

\newcommand{\m}{\ensuremath{\mathfrak{m}}\xspace}

\newcommand{\OO}{\ensuremath{\mathcal{O}}\xspace}
\newcommand{\U}{\ensuremath{\mathcal{U}}\xspace}

\newcommand{\Frob}{\mathrm{Frob}\xspace}

\newcommand{\oomment}[1]{}

\DeclareMathOperator{\Gal}{Gal}
\DeclareMathOperator{\End}{End}
\DeclareMathOperator{\Hom}{Hom}

\DeclareMathOperator{\Sym}{Sym}
\DeclareMathOperator{\Spec}{Spec}

\DeclareMathOperator{\Spa}{Spa}

\DeclareMathOperator{\Fil}{Fil}

\DeclareMathOperator{\Res}{Res}

\newcommand{\rhobar}{\ensuremath{\overline{\rho}}\xspace}
\newcommand{\T}{\ensuremath{\mathbb{T}}\xspace}

\newcommand{\GL}{\ensuremath{\mathrm{GL}}\xspace}

\newcommand{\cW}{\mathcal{W}\xspace}
\newcommand{\cZ}{\mathcal{Z}\xspace}

\newcommand{\mb}{\mathbb}
\newcommand{\mc}{\mathcal}

\newcommand{\mf}{\mathfrak}
\newcommand{\vp}{\varpi}
\newcommand{\ra}{\rightarrow}

\newcommand{\sub}{\subseteq}
\newcommand{\oo}{\mathcal{O}}
\newcommand{\ol}{\overline}

\newcommand{\ka}{\kappa}

\newcommand{\wt}{\widetilde}

\newtheorem{theorem}{Theorem}[subsection]
\newtheorem{proposition}[theorem]{Proposition}
\newtheorem{corollary}[theorem]{Corollary}
\newtheorem{lemma}[theorem]{Lemma}

\newtheorem{conjecture}[theorem]{Conjecture} 
\newtheorem*{theor}{Theorem}

\theoremstyle{definition}
\newtheorem{definition}[theorem]{Definition}

\newtheorem{remark}[theorem]{Remark}

\mathchardef\mhyphen="2D
\title{Parallel weight $2$ points on Hilbert modular eigenvarieties and the 
parity conjecture}
\author{Christian Johansson and James Newton}
\date{\today}
\input xy
\xyoption{all}
\usepackage{amscd}

\begin{document}
\begin{abstract}
Let $F$ be a totally real field and let $p$ be an odd prime which is totally 
split in 
$F$. We define and study one-dimensional `partial' eigenvarieties interpolating 
Hilbert modular forms over $F$ with weight varying only at a single place $v$ 
above $p$. For these eigenvarieties, we show that methods developed by Liu, 
Wan 
and Xiao apply and deduce that, over a boundary annulus in weight space of 
sufficiently 
small radius, the partial eigenvarieties decompose as a disjoint union of 
components which are finite 
over weight space. We apply this result to prove the parity version of the 
Bloch--Kato conjecture for finite slope Hilbert modular forms with trivial central 
character (with a technical assumption if $[F:\Q]$ is odd), by reducing to the 
case of parallel weight $2$. As another 
consequence of our results on partial eigenvarieties, we show, still under the 
assumption that $p$ is totally split in $F$, that the `full' (dimension 
$1+[F:\Q]$) cuspidal Hilbert modular eigenvariety has the property that many (all, if $[F:\Q]$ is even) 
irreducible components contain a classical point with noncritical slopes and 
parallel weight $2$ (with 
some character at $p$ whose conductor can be explicitly bounded), or any other 
algebraic weight.
\end{abstract}
\maketitle

\counterwithin{equation}{subsection}
\section{Introduction}
\subsection{Eigenvarieties near the boundary of weight space}
In recent work, Liu, Wan and Xiao \cite{lwx} have shown that, over a boundary 
annulus in 
weight space of sufficiently small radius, the eigencurve for a definite 
quaternion algebra over $\Q$ decomposes as a disjoint union of components which 
are finite over weight space. On each component, the slope (i.e.~the $p$-adic 
valuation of the $U_p$-eigenvalue) varies linearly with the weight, and in 
particular the slope tends to zero as the weight approaches the boundary of the 
weight space. A notable consequence of this result is that every irreducible 
component of the eigencurve contains a classical weight $2$ point (with 
$p$-part of the Nebentypus character of large conductor, so that the 
corresponding weight-character is close to the boundary of weight space).

The purpose of this note is to extend the methods of \cite{lwx} to establish a 
similar result for certain one-dimensional eigenvarieties 
interpolating Hilbert modular forms (Proposition \ref{lwx}). More precisely, we 
assume that the (odd) 
prime $p$ splits completely in a totally real field $F$, and consider 
`partial eigenvarieties' whose classical points are automorphic forms for 
totally definite quaternion algebras over $F$ whose weights are only allowed to 
vary at 
a single place $v|p$.\footnote{One can also consider 
the slightly more general situation where the assumption is only that the place 
$v$ is 
split.} 

We give two applications of 
this result. The first application, following 
methods and results of Nekov\'{a}\v{r} (for example, \cite{nekparIII}) and 
Pottharst--Xiao 
\cite{px}, is to 
establish new cases of the parity part of the Bloch--Kato conjecture for the 
Galois representations associated to Hilbert modular forms (see the next 
subsection for a precise statement). The idea of the proof is that, using the 
results of \cite{px}, we can reduce the parity conjecture for a Hilbert 
modular newform $g$ of general (even) weight to the parity conjecture for 
parallel weight two Hilbert modular forms, by moving in a $p$-adic family 
connecting $g$ to a parallel weight two form. This parallel weight two form 
will have local factors at places 
dividing $p$ given by ramified principal series representations (moving to the 
boundary of weight space corresponds to increasing the conductor of the ratio 
of the characters defining this principal series representation).  Using our 
results on 
the partial 
eigenvarieties, we carry out this procedure in $d = [F:\Q]$ steps, where each 
step moves one of the weights to two. An additional difficulty when $F\neq \Q$ is to ensure that at 
each step we move to a point which is noncritical at all the places 
$v^{\prime}|p$, so 
that the global triangulation over 
the family provided by the results of \cite{kpx} specialises to a 
triangulation at all of these points. At the fixed $v^{\prime}=v$, this follows 
automatically from the construction. At $v^{\prime}\neq v$ we show, roughly 
speaking, that the locus of points on the partial eigenvariety which are 
critical at $v^{\prime}$ forms a union of connected components, so one cannot 
move from noncritical to critical points.

The second application (see Section \ref{sec:fullevar}) is to establish that 
every irreducible component of the 
`full' eigenvariety (with dimension $[F:\Q]+1$) for a totally definite 
quaternion 
algebra over $F$ contains a classical point with noncritical slopes and 
parallel weight $2$, or any 
other 
algebraic weight.

\subsection{The Bloch--Kato conjecture}
Let $g$ be a normalised cuspidal Hilbert modular newform of weight 
$(k_1,k_2,\ldots,k_d,w=2)$ and level $\Gamma_0(\mf{n})$, over 
a totally real number field $F$, with each $k_i$ even. 
We suppose moreover that the automorphic 
representation associated to $g$ has trivial central character.\footnote{Here 
we use the notation and conventions of \cite[\S 1]{saito} for the weights of 
Hilbert modular forms, and in particular $w=2$ corresponds to the central 
character of the associated automorphic representation having trivial algebraic 
part. Note that in the body of the text this will correspond to taking $w=0$ in 
the weights we define for totally definite quaternion algebras.} 

Let $E$ be the (totally real) number field generated by the Hecke eigenvalues 
$t_v(g)$. For any finite place $\lambda$ of 
$E$, with residue characteristic $p$, denote by $V_{g,\lambda}$ the 
two-dimensional (totally odd, absolutely irreducible) 
$E_\lambda$-representation associated to $g$ which satisfies
\[\det(X-\Frob_v|_{V_{g,\lambda}}) = X^2 - t_v(g) X +q_v\] for all $v\nmid 
\mf{n}p$ ($\Frob_v$ denotes arithmetic Frobenius).

Note that we have $V_{g,\lambda} \cong V_{g,\lambda}^*(1)$. The conjectures of 
Bloch and Kato predict the following formula relating the 
dimension of the Bloch--Kato Selmer group of $V_{g,\lambda}$ to the central 
order of vanishing of an $L$-function:

\begin{conjecture}
	\[\dim_{E_\lambda} H^1_f(F,V_{g,\lambda}) = 
	\mathrm{ord}_{s=0}L(V_{g,\lambda},s) \left(\mathrm{or, equivalently, }= 
	\mathrm{ord}_{s=0}L(g,s)\right).\]
\end{conjecture}
In this conjecture, the $L$-functions are normalised so that the local 
factors (at good 
places) are \[L_v(V_{g,\lambda},s) = L_v(g,s) = 
\left(1-t_vq_v^{-s-1}+q_v^{-2s-1}\right)^{-1}.\]

We refer to the following as the `parity conjecture for $V_{g,\lambda}$':

\begin{conjecture}
	\[\dim_{E_\lambda} H^1_f(F,V_{g,\lambda}) \equiv 
	\mathrm{ord}_{s=0}L(V_{g,\lambda},s) 
	\pmod{2}.\]
\end{conjecture}

Our main result towards the parity conjecture is the following, which is proved 
in Section \ref{sec:parity}:

\begin{theorem}\label{parityapp}
	Let $p$ be an odd prime and let $F$ be a totally real number field in which 
	$p$ splits completely. Let $g$ be a cuspidal Hilbert modular newform for 
	$F$ 
	with weight 
	$(k_1,k_2,\ldots,k_d,2)$ (each $k_i$ even) such that the associated 
	automorphic 
	representation $\pi$ has trivial central character. Let $E$ be the number 
	field generated by the Hecke eigenvalues of $g$ and let $\lambda$ be a 
	finite place of 
	$E$ with residue characteristic $p$.
	
	Suppose that for every place $v|p$, $\pi_v$ is not supercuspidal (in other 
		words, $g$ has finite slope, up to twist, for every $v|p$). If $[F:\Q]$ 
		is odd 
		suppose moreover that there is a finite place $v_0 
	\nmid p$ of $F$ such that $\pi_{v_0}$ is not 
	principal series. Then the parity conjecture holds for $V_{g,\lambda}$.
\end{theorem}
\begin{remark}
\begin{enumerate}
	\item For fixed $g$ there is a positive density set of primes $p$ 
	satisfying the 
	assumptions of the above theorem. So combining our theorem with the 
	conjectural `independence of $p$' for the (parity of the) rank of the 
	Selmer group  $H^1_f(F,V_{g,\lambda})$ would imply the parity conjecture 
	for the entire compatible family of Galois representations associated to 
	$g$ (with the assumption on the existence of a suitable $v_0$ if $[F:\Q]$ 
	is 
	odd).
	\item When $k_i = 2$ for all $i$, the above theorem is already contained in 
	\cite{px}. Moreover, stronger results (without the finite slope hypothesis) 
	are given by \cite[Thm.~1.4]{nekmr} and \cite[Thm.~C]{nektame}.  As far as 
	the authors 
	are aware, the only prior results for higher weights are for $p$-ordinary 
	forms (for example, \cite[Thm.~12.2.3]{nekselmer}), except when $F=\Q$ in 
	which case the above theorem follows from the results of \cite{px} and 
	\cite{lwx} (cf.~\cite[Remark 1.10]{lwx}), and there is also a result of 
	Nekov\'{a}\v{r} which applies under a mild technical hypothesis 
	\cite[Thm.~B]{nekmr}.
	\item When $[F:\Q]$ is odd we require the existence of the place $v_0$ in 
	the above theorem in order to switch to a totally definite quaternion 
	algebra. 
\end{enumerate}
\end{remark}

\subsection*{Acknowledgments} The authors wish to thank Liang Xiao for useful 
discussions on the paper \cite{px}, together with John Bergdall, Ga\"{e}tan 
Chenevier and Jan 
Nekov\'{a}\v{r} for helpful 
correspondence. We also thank an anonymous referee for helpful comments. 
C.J.~was supported by NSF grant DMS 1128155 during the initial 
stages of this work.

\section{A halo for the partial eigenvariety}\label{sec:evar}
\subsection{Notation}
We fix an odd prime number $p$\footnote{The assumption that $p$ is odd is not 
essential for this section, but it appears in the results of \cite{px} and 
\cite{nektame} which we 
will apply later, so we have excluded the case $p=2$ throughout this paper  for 
simplicity.}, a totally real number field $F$ with $[F:\Q] = 
d > 1$, and we assume that $p$ splits completely in $F$. We also fix a totally 
definite quaternion 
algebra $D$ over $F$, with discriminant $\delta_D$, and assume that $p$ is 
coprime to $\delta_D$. We 
fix a maximal order $\oo_D \subset D$ and an isomorphism 
$\oo_{D}\otimes_{\oo_F}\widehat{\OO}_F^{\delta_D} \cong 
M_2(\widehat{\OO}_F^{\delta_D}) = \prod_{v\nmid \delta_D}M_2(\oo_{F,v})$. 
Using this isomorphism, we 
henceforth identify $\oo_{D,p}^\times$ and $\prod_{v|p}\GL_2(\oo_{F,v})$. We 
fix the uniformiser $\varpi_v=p$ of $\oo_{F,v}$ for each $v| p$.

For $v|p$, we let $T_{0,v} = \begin{pmatrix}
\oo_{F,v}^\times & 0 \\ 0 & \oo_{F,v}^\times
\end{pmatrix} \subset \GL_2(\oo_{F,v})$, let \begin{align*}I_{v,n} &= 
\left\{\begin{pmatrix}
a & b \\ \varpi_v^nc & d
\end{pmatrix}  \mid a,b,c,d \in \oo_{F,v} \right\}\cap 
\GL_2(\oo_{F,v}),\end{align*} 

\begin{align*}\ol{I}_{v,n} &= 
\left\{\begin{pmatrix}
a & \varpi_v^nb \\ c & d
\end{pmatrix}  \mid a,b,c,d \in \oo_{F,v} \right\}\cap 
\GL_2(\oo_{F,v})\end{align*}

and \begin{align*}
I_{v,1,n} &= \left\{\begin{pmatrix}
	a & b \\ \varpi_vc & d
\end{pmatrix} \in I_{v,1}  \mid b\equiv 0 
\pmod{\varpi_v^n}\right\}, \\ I'_{v,1,n} &= \left\{\begin{pmatrix}
a & \varpi_v^nb \\ \varpi_vc & d
\end{pmatrix} \in I_{v,1,n}  \mid a\equiv d \equiv 1 
\pmod{\varpi_v^{n+1}}\right\}.
\end{align*}

Set \[\Sigma_v^+ = \left\{\begin{pmatrix}
1 & 0 \\ 0 & \vp_v^\beta
\end{pmatrix}  \mid \beta \in \Z_{\geq 0} \right\}\] and \[\Sigma_v^{cpt} = 
\left\{\begin{pmatrix}
1 & 0 \\ 0 & \vp_v^\beta
\end{pmatrix}  \mid \beta \in \Z_{>0} \right\}.\]

Fix a place $v|p$ (this will be the place where the weight varies in our 
families of $p$-adic automorphic forms), and let $K^v = \prod_{v' \ne v}K_{v'}$ 
be a compact open 
subgroup of $(\oo_D\otimes_F \A_{F,f}^{v})^\times$ such that $K = 
K^vI_{v,1}$ is a \emph{neat} subgroup of $(D\otimes_F \A_{F,f})^\times$. More 
precisely, we assume that $x^{-1}D^\times x \cap K \subset 
\oo_F^{\times,+}$ (the totally positive units)
for all $x \in (D\otimes_F\A_{F,f})^\times$. 

We make the following general definitions:

\begin{definition}
	Let $K$ be a compact open subgroup of $(D\otimes_F \A_{F,f})^\times$  and let $N$ be a monoid with $K 
		\subset N \subset (D\otimes_F \A_{F,f})^\times$. Suppose $M$ is a left $R[N]$-module (for some commutative coefficient ring 
	$R$). 
	
	\begin{enumerate}
		\item If $f: D^\times\backslash (D\otimes_F \A_{F,f})^\times 
		\rightarrow M$ is a function and $\gamma \in N$ we define a function $_\gamma|f$ by $_\gamma|f(g) = 
		\gamma f(g\gamma)$.
		\item We define \[H^0(K,M) = \{f: D^\times\backslash (D\otimes_F 
		\A_{F,f})^\times 
		\rightarrow M \mid ~_u|f= f \text{ for all } u \in K  \}.	
		\]
		\item If $K' \subset K$ is a 
		compact open subgroup then we can define a double coset operator 
		\[[KgK']: H^0(K',M) \rightarrow H^0(K,M)\] for any $g \in N$ by 
		decomposing 
		the double coset $KgK' = \coprod_{i}g_iK'$ and defining \[[KgK']f = 
		{\sum_i}~_{g_i}|f.\]
	 \end{enumerate}
\end{definition}

The following lemma is easily checked.
\begin{lemma}\label{neatdecomp}
	Suppose $K$ is a neat compact open subgroup of $(D\otimes_F 
	\A_{F,f})^\times$. Let $g_1, \ldots, g_t$ be double coset representatives 
	for $D^\times\backslash (D\otimes_F \A_{F,f})^\times / K$, and let $M$ be 
	an $R[K]$-module on which $Z(K) := \oo_F^{\times}\cap K = 
	\oo_F^{\times,+}\cap K$ acts trivially (the equality follows from the 
	neatness assumption).
	
	Then the $R$-module map \[f \mapsto (f(g_1),\ldots,f(g_t))\] gives an 
	isomorphism $H^0(K,M) \cong M^{\oplus t}$.
\end{lemma}

Now suppose we have an integer $w$ and a tuple of integers \[k^v = 
(k_{v'})_{v'|p, v' \ne v}\]
such that $k_{v'} \equiv w \pmod 2$ and $k_{v'} \ge 2$ for all $v'$.

For each $v'$ we define a finite free $\Z_p = \oo_{F_{v'}}$-module with a left 
action 
of $K_{v'}$: \[\mc{L}_{v'}(k_{v'},w) = 
\mathrm{det}^{(w-k_{v'}+2)/2}\Sym^{k_{v'}-2} 
\oo_{F_{v'}}^2\] where $\oo_{F_{v'}}^2$ is the standard representation of 
$\GL_2(\oo_{F,v'})\cong \oo_{D,v'}^\times$.

We then define a finite free $\Zp$-module with a left action of $K^v$: 
\[\mc{L}^v(k^v,w) := \bigotimes_{v'|p, v'\ne v}\mc{L}_{v'}(k_{v'},w)\] 

Let $\chi^v$ be 
the 
character of $T_0^v = \prod_{v'|p, v'\ne v}T_{0,v'}$ given by the highest 
weight of $\mc{L}^v(k^v,w)$ (with respect to the upper triangular Borel). 
Explicity, this character is  
\[\chi^v(t_1,t_2) := 
\prod_{v'}(t_{1,v'})^{\frac{w+k_{v'}-2}{2}}(t_{2,v'})^{\frac{w-k_{v'}+2}{2}}\]

We now recall some terminology from \cite{extended}. We let $R$ be a 
Banach--Tate $\Z_p$-algebra \cite[3.1]{extended} with a multiplicative 
pseudo-uniformiser $\vp$, and let $\kappa: T_{0,v} \ra R^\times$ be a 
continuous character. Assume that the norm on $R$ is adapted to $\kappa$ 
\cite[Def.~3.3.2]{extended} and that $\chi^v\kappa$ is trivial on $Z(K)$. We 
call such a $\kappa$ a \emph{weight}.

For $r \ge r_\kappa$ we get a Banach $R$-module of 
distributions $\mc{D}_{\ka}^r$ \cite[Def.~3.3.9]{extended} equipped with a left 
action of the monoid 
$\Delta_v = I_{v,1}\Sigma_v^+I_{v,1}$.  This is a space of distributions, and we also have (see the end of \cite[\S 3.3]{extended}) 
a 
space of functions $\mc{A}_{\ka}^{<r}$ on $\begin{pmatrix}
1 & 0 \\ \varpi_v\OO_{F,v} & 1
\end{pmatrix}$ (which we identify with functions in a single variable $x$ on 
$\varpi_v\OO_{F,v}$) with a right action of $\gamma=\begin{pmatrix}
a & b\\c & d
\end{pmatrix} \in \Delta_v$ given by \[(f\cdot\gamma)(x) = 
\kappa(a+bx,(\det\gamma)|\det\gamma|_{v}/(a+bx))f\left(\frac{c+dx}{a+bx}\right).\]

There is a natural pairing $\mc{D}_{\ka}^{r}\times \mc{A}_{\ka}^{<r} \ra R$ which identifies $\mc{A}_{\ka}^{<r}$ as the dual of $\mc{D}_{\ka}^{r}$ (and not the other way around!), and hence gives an embedding of $\mc{D}_{\ka}^{r}$ into the dual of $\mc{A}_{\ka}^{<r}$. The $R$-module $H^0(K,\mc{L}^v(k^v,w)\otimes_{\Zp}\mathcal{D}^{r}_{\kappa})$ is 
the space of $p$-adic automorphic forms with fixed weights away from $v$ which 
we will be studying in this section.

We have a natural action of 
$(D\otimes_F\A_{F,f}^{p})^\times \times \Delta_v$ on 
$\mc{L}^v(k^v,w)\otimes_{\Zp}\mathcal{D}^{r}_{\kappa}$, and we get associated 
double coset operators. If $R$ is a $\Qp$-algebra, we can extend this to an 
action of $(D\otimes_F\A_{F,f}^{v})^\times \times \Delta_v$. We will especially 
consider the action on 
$H^0(K,\mc{L}^v(k^v,w)\otimes_{\Zp}\mathcal{D}^{r}_{\kappa})$ of the  
Hecke 
operator $U_v = [K\begin{psmallmatrix} 1 & 0 \\ 0 & 
\varpi_v \end{psmallmatrix}K]$. This Hecke operator is compact because, by 
\cite[Cor.~3.3.10]{extended}, it 
factors as a composition  
\[H^0(K,\mc{L}^v(k^v,w)\otimes_{\Zp}\mathcal{D}^{r}_{\kappa})\rightarrow 
H^0(K,\mc{L}^v(k^v,w)\otimes_{\Zp}\mathcal{D}^{r^{1/p}}_{\kappa})\hookrightarrow
H^0(K,\mc{L}^v(k^v,w)\otimes_{\Zp}\mathcal{D}^{r}_{\kappa})\] where the second 
map is induced by the natural compact inclusion 
$\mathcal{D}^{r^{1/p}}_{\kappa} \hookrightarrow \mathcal{D}^{r}_{\kappa}$.

When $R$ is a $\Qp$-algebra we will also use the (non-compact) Hecke operators 
$U_{v'}= 
p^{-(w-k_{v'}+2)/2}[K\begin{psmallmatrix} 1 & 0 \\ 
0 & 
\varpi_{v'} \end{psmallmatrix}K]$ for the 
places $v \ne v'|p$. We have 
normalised the operators $U_{v'}$ so that they are consistent with $U_v$. In 
particular, when $R$ is a field extension of $\Qp$, the non-zero eigenvalues 
for $U_{v'}$ will have $p$-adic valuation between 
$0$ and $k_{v'}-1$.

\subsection{Locally algebraic weights}\label{ssec:lalgwts}
Suppose $\kappa: T_{0,v} \rightarrow E^\times$ is a weight, where $E$ is a 
finite extension of $\Qp$, and for some positive integer $k$ with $k\equiv w 
\pmod{2}$ the restriction of 
$\kappa$ to an open subgroup $(1+\varpi_v^n\oo_{F_v})^2$ of $T_{0,v}$ 
coincides with the character
\[\chi_k(t_1,t_2) = 
(t_{1})^{\frac{w+k-2}{2}}(t_{2})^{\frac{w-k+2}{2}}.\] Then, denoting the finite 
order character $\kappa/\chi_k$ by 
$\epsilon$, we say that 
$\kappa$ is \emph{locally algebraic of weight} $(k,w)$ \emph{and character} 
$\epsilon$. 
If we make the 
standard identification of the $E$-dual (right) representation 
$\mc{L}_{v}(k,w)^\vee$ with homogeneous degree $k$ polynomials in two 
variables 
$(X,Y)$ then evaluation at $(1,x)$ gives an injective 
$I'_{v,1,n-1}$-equivariant 
map \[ \mc{L}_{v}(k,w)^\vee \rightarrow \mc{A}_{\ka}^{<r}.\] The action of 
$\begin{psmallmatrix}
1 & 0\\ 0 & \varpi_v
\end{psmallmatrix}$ on the right hand side induces the action of 
$p^{-(w-k+2)/2}\begin{psmallmatrix}
1 & 0\\ 0 & \varpi_v
\end{psmallmatrix}$ on the left hand side. Taking duals and using the embedding $\mc{D}_{\ka}^{r}\hookrightarrow (\mc{A}_{\ka}^{<r})^{\vee}$, we 
obtain a surjective $I'_{v,1,n-1}$-equivariant map \[\mc{D}_{\ka}^r \rightarrow 
\mc{L}_{v}(k,w)\otimes_{\Zp}E.\] 

Moreover, if we let $\mc{L}_{v}(k,w,\epsilon)$ denote the 
$I_{v,1,n-1}$-representation obtained from $\mc{L}_{v}(k,w)\otimes_{\Zp}E$ by 
twisting the 
action by $\epsilon$ then we obtain a surjective $I_{v,1,n-1}$-equivariant map 
\[\mc{D}_{\ka}^r \rightarrow 
\mc{L}_{v}(k,w,\epsilon).\] 

\subsection{Comparing level $I_{v,1}$ and $I_{v,1,n-1}$}We retain the set-up of 
the previous subsection, so $\kappa$ is a locally algebra weight. A version of 
the following lemma is commonly used in Hida theory (see for 
example 
\cite[Lem.~2.5.2]{ger}):
\begin{lemma}\label{lem:fslevelindep}
	\leavevmode
	\begin{enumerate}
		\item For $n \ge 2$, the endomorphism $U_v$ of 
	$H^0(K^vI_{v,1,n-1},\mc{L}^v(k^v,w)\otimes_{\Zp}\mathcal{D}^{r}_{\kappa})$
	 factors through the natural inclusion 
	\[H^0(K^vI_{v,1,n-2},\mc{L}^v(k^v,w)\otimes_{\Zp}\mathcal{D}^{r}_{\kappa})
	 \hookrightarrow  
	H^0(K^vI_{v,1,n-1},\mc{L}^v(k^v,w)\otimes_{\Zp}\mathcal{D}^{r}_{\kappa}).\]
	\item The natural inclusion 	
	\[\iota: 
	H^0(K^vI_{v,1},\mc{L}^v(k^v,w)\otimes_{\Zp}\mathcal{D}^{r}_{\kappa})
	\hookrightarrow  
	H^0(K^vI_{v,1,n-1},\mc{L}^v(k^v,w)\otimes_{\Zp}\mathcal{D}^{r}_{\kappa})\] 
	is $U_v$-equivariant and if $h \in \Q_{\ge 0}$ it induces an isomorphism 
	between $U_v$-slope $\le h$ subspaces:
	\[\iota: 
	H^0(K^vI_{v,1},\mc{L}^v(k^v,w)\otimes_{\Zp}\mathcal{D}^{r}_{\kappa})^{\le 
	h}
	\cong  
	H^0(K^vI_{v,1,n-1},\mc{L}^v(k^v,w)\otimes_{\Zp}\mathcal{D}^{r}_{\kappa})^{\le
	 h}\]
	 	\end{enumerate}
\end{lemma}
\begin{proof}
First we note that (for $n\geq 1$) the action of $U_v$ on 
$H^0(K^vI_{v,1,n-1},\mc{L}^v(k^v,w)\otimes_{\Zp}\mathcal{D}^{r}_{\kappa})$ is 
given by \[f \mapsto \sum_{i=0}^{p-1} ~_{\begin{psmallmatrix}
1 & 0\\ \varpi_vi & \varpi_v
\end{psmallmatrix}}|f.\] The $U_v$-equivariance in the second statement follows 
immediately from this. The rest of the second statement follows from the first, 
since $U_v$ acts invertibly on the slope $\le h$ subspace.

It follows from a simple computation that the double coset operator 
$[I_{v,1,n-1}\begin{psmallmatrix}
1 & 0\\ 0 & \varpi_v
\end{psmallmatrix} I_{v,1,n-1}]$ is equal to $[I_{v,1,n-2}\begin{psmallmatrix}
1 & 0\\ 0 & \varpi_v
\end{psmallmatrix}I_{v,1,n-1}]$, and the first part follows immediately from 
this.
\end{proof}

Having done all this, composing $\iota$ with the map induced by 
$\mc{D}_{\ka}^{r} \rightarrow 
\mc{L}_{v}(k,w,\epsilon)$,  we obtain a Hecke-equivariant map 

\begin{equation}\label{maptoclass}\pi: 
H^0(K^vI_{v,1},\mc{L}^v(k^v,w)\otimes_{\Zp}\mathcal{D}^{r}_{\kappa}) 
\rightarrow  
H^0(K^vI_{v,1,n-1},\mc{L}^v(k^v,w)\otimes_{\Zp}\mc{L}_v(k,w,\epsilon)),\end{equation}
where the action of $U_v$ on the target is the $\star$-\emph{action} defined by 
multiplying the standard action of $U_v$ by $p^{-(w-k+2)/2}$.

\begin{proposition}\label{partialclass}
	Let $h \in \Q_{\ge 0}$ with $h < k-1$. The map $\pi$ induces an 
	isomorphism between $U_v$-slope $\le h$ subspaces. 
\end{proposition}
\begin{proof}
 By Lemma \ref{lem:fslevelindep} it suffices to show that the map \[ 
 H^0(K^vI_{v,1,n-1},\mc{L}^v(k^v,w)\otimes_{\Zp}\mathcal{D}^{r}_{\kappa}) 
 \rightarrow  
 H^0(K^vI_{v,1,n-1},\mc{L}^v(k^v,w)\otimes_{\Zp}\mc{L}_v(k,w,\epsilon))\] 
 induces an isomorphism between $U_v$-slope $\le h$ subspaces, and this can be 
 proved as in \cite[Thm.~3.2.5]{han}.
\end{proof}

\subsection{The `Atkin--Lehner trick'}\label{ss:altrick}
In this section we establish a result analogous to \cite[Prop.~3.22]{lwx}. Let 
$n \ge 2$ be an integer and suppose \[K_{v'} \subset \begin{pmatrix}
1+\varpi_{v'}^n\oo_{F,v'} & \oo_{F,v'} \\ \varpi_{v'}^n\oo_{F,v'} & 
1+\varpi_{v'}^n\oo_{F,v'}
\end{pmatrix}\] for each $v'|p$, $v' \ne v$. Note that, combined with our 
neatness assumption, this implies that $Z(K)$ is 
contained in $1 + 
p^n\OO_F$. 

Fix an integer $k \ge 2$ with 
the 
same parity as $w$, together with a finite order character $\epsilon = 
(\epsilon_1,\epsilon_2): 
T_{0,v} \ra 
E^\times$, where $\epsilon_1$ and $\epsilon_2$ are characters of 
$(\OO_{F,v}/\varpi_v^n)^\times$ and $\epsilon_2/\epsilon_1$ has conductor 
$(\varpi_{v}^n)$.

Let $\epsilon_\Q$ be the finite order 
Hecke 
character of $\Q^\times\backslash\A_\Q^\times$ associated to the Dirichlet 
character \[(\Z/p^n\Z)^\times \cong (\oo_{F,v}/\vp_v^n)^\times 
\overset{\epsilon_1 \epsilon_2}{\rightarrow} E^\times.\]

We now consider the space of classical automorphic forms 
$S(k,w,\epsilon):=H^0(K^vI_{v,1,n-1},\mc{L}^v(k^v,w)\otimes 
\mc{L}(k,w,\epsilon))$. By Lemma \ref{neatdecomp}, the dimension of 
this 
space is equal to 
$(k-1)p^{n-1}t$ (as we noted above, $Z(K) \subset 1 + 
p^n\OO_F$, so it acts trivially on the coefficients), where \[t = 
|D^\times\backslash 
(D\otimes_F\A_{F,f})^\times/K^vI_{v,1}|\prod_{v'|p,v'\ne v}(k_{v'}-1).\]

Denote the slopes of $U_v$ appearing in $S(k,w,\epsilon)$ by 
$\alpha_0(\epsilon),\ldots,\alpha_{(k-1)p^{n-1}t-1}(\epsilon)$ in 
non-decreasing order. 

\begin{lemma}\label{AtkinLehner}
	We have $\alpha_i(\epsilon) = k-1 - 
	\alpha_{(k-1)p^{n-1}t-1-i}({\epsilon}^{-1})$ 
	for $i = 0,\ldots,(k-1)pt-1$. In particular, the sum (with 
	multiplicities) 
	of the $U_v$ slopes appearing in $S(k,w,\epsilon) \oplus 
	S(k,w,{\epsilon}^{-1})$ is $(k-1)^2p^{n-1}t$.
\end{lemma}
\begin{proof}
First we fix an embedding $\iota: E \hookrightarrow \mathbb{C}$. Using this 
embedding, we regard $\epsilon_\Q$ and $\epsilon$ as complex valued characters. 
The space  
$S(k,w,\epsilon)\otimes_{E,\iota}\mathbb{C}$ can be described in terms of 
automorphic representations for $D$. If $\pi$ contributes to this space,  
the local factor $\pi_v$ is a principal 
series representation of $\GL_2(F_v)$ obtained as the normalised induction of a 
pair of characters \[\mathrm{unr}(\alpha^{-1}\zeta p^{w})\epsilon_1^{-1}\times
\mathrm{unr}(\alpha)\epsilon_2^{-1}\] where $\zeta$ is a 
root of unity. To see 
this, we first note that $\pi_v$ has a non-zero subspace on which $I_{v,1,n-1}$ 
acts via the character $\epsilon^{-1}$ (these are the vectors in 
$\pi_v$ which contribute 
to $S(k,w,\epsilon)\otimes_{E,\iota}\mathbb{C}$). So $\pi_v 
\otimes \epsilon_2\circ\det$ has a non-zero subspace on which $I_{v,1,n-1}$ 
acts via 
$\begin{psmallmatrix}
a & b\\ c & d
\end{psmallmatrix} \mapsto (\epsilon_2/\epsilon_1)(a)$. Applying  
$\begin{pmatrix}
1 & 0\\ 0 & \varpi_v^{n-1}
\end{pmatrix}$ to this subspace, we get a non-zero subspace where the action of 
$I_{v,n}$ is given by the same formula. Therefore the 
conductors of both $\pi_v\otimes \epsilon_2\circ\det$ and its central character 
$\epsilon_2/\epsilon_1$ are $(\varpi_{v
}^{n})$. It follows (e.g.~by 
\cite[Lem.~3.3]{templier}) that $\pi_v\otimes \epsilon_2\circ\det$ is the 
normalised induction of $\mu_1\times\mu_2$ with $\mu_1|_{\OO_{F,v}^\times} = 
\epsilon_2/\epsilon_1$ and $\mu_2|_{\OO_{F,v}^\times} = 1$. The rest of the 
claim is deduced from the fact that the central character of $\pi_v$ is the
product of a finite order unramified character and
$\epsilon_1^{-1}\epsilon_2^{-1}|\cdot|^{-w}_v$.

We now need to compute the eigenvalue for the standard $U_v$ action on 
the subspace of $\pi_v$ where $I_{v,1,n-1}$ acts via $\begin{psmallmatrix}
a & b\\ c & d
\end{psmallmatrix} \mapsto \epsilon_1(a)^{-1}\epsilon_2(d)^{-1}.$ Conjugating 
by $\begin{psmallmatrix}
1 & 0\\0 & \vp
\end{psmallmatrix}$, we can do the same for the standard $U_v$ action defined 
with respect to $\ol{I}_{v,n}$, and we get 
$U_v$-eigenvalue $\alpha p^{{1}/{2}}$. Using the $\star$-action we therefore
get $U_v$-eigenvalue $\alpha p^{\frac{k-1-w}{2}}$. 

Twisting by 
$\epsilon_\Q$ gives a $\pi'$ with local factor $\pi'_v$ the 
normalised induction of \[\mathrm{unr}(\alpha^{-1}\zeta p^{w})\epsilon_2\times
\mathrm{unr}(\alpha)\epsilon_1,\] which contributes to 
$S(k,w,{\epsilon}^{-1})$  with $\star$-action $U_v$-eigenvalue 
$\alpha^{-1}\zeta p^{\frac{w+k-1}{2}}$. 
Summing the two slopes together we get 
$(k-1)$, which gives the desired result.
\end{proof}

\subsection{The weight space of the partial eigenvariety}\label{ssec:partial}
We are going to construct a `partial eigenvariety' out of the spaces 
$H^0(K,\mc{L}^v(k^v,w)\otimes_{\Zp}\mathcal{D}^{r}_{\kappa})$. The underlying weight space $\cW$ is defined by letting \[\cW(A) = \{\kappa \in 
\Hom_{cts}(T_{0,v},A^\times) : \chi^v\kappa|_{Z(K)} = 1\}\] for
algebras $A$ of topologically finite type over $\Qp$. We let $\Delta$ denote the torsion subgroup of $\OO_{F,v}^\times$.

\begin{lemma}\label{wtcomps}
	We have an isomorphism $\cW_{\C_p} \cong \coprod_{\eta,\omega} 
	D_{\mathbb{C}_p}$ where $D$ 
	is the 
	open unit disc and $\eta, \omega$ run over pairs of characters $\eta: 
	\OO_v^\times \rightarrow \C_p^\times, \omega: \Delta \times \Delta 
	\rightarrow 
	\C_p^\times$ such that $\chi^v\eta|_{Z(K)} = 1$ and $\omega$ restricted to 
	the diagonal copy of $\Delta$ is equal to $\eta|_\Delta$.
	
	The isomorphism is given by taking $\kappa$ to 
	$\kappa(\exp(p),\exp(p)^{-1})-1$ in the component labelled by $\eta(t_v) = 
	\kappa(t_v,t_v)$ and $\omega(\delta_1,\delta_2) = 
	\kappa(\delta_1,\delta_2)$.\end{lemma}
\begin{proof}
	This follows from the fact that the closure of $Z(K)$ in $T_{0,v}$ 
	is a finite index subgroup of $\OO_v^\times$ centrally embedded in 
	$T_{0,v}$.
\end{proof} \begin{remark}\label{finordchar}
The condition that $\chi^v\eta|_{Z(K)} = 1$ is equivalent 
to the condition that $\eta(t_v) = t_v^w\tilde{\eta}(t_v)$ where $\tilde{\eta}$ 
is a finite order character, trivial on $Z(K)$. In particular, there are 
finitely many possibilities for $\eta$. Moreover, for each $\eta$, if we denote 
by $E_\eta$ the finite extension of $\Qp$ generated by the values of $\eta$  
then the 
open and closed immersion $\coprod_{\omega}D_{\mathbb{C}_p} \hookrightarrow 
\cW_{\C_p}$ given by restricting to the components labelled by $\eta$ is 
defined over $E_\eta$. For each pair $\eta, \omega$ we therefore denote by 
$\cW_{\eta,\omega}$ the corresponding connected component (which is isomorphic 
to the open unit disc over $E_\eta$) of $\cW_{E_\eta}$. 
\end{remark}

If $\kappa$ is a point of $\cW_{\mathbb{C}_p}$ we denote by $z_\kappa$ the 
corresponding point 
of $D_{\mathbb{C}_p}$. If $r \in (0,1)\cap\Q$, the union of open annuli given 
by $r < 
|z_\kappa|_p < 
1$ is denoted 
by $\cW^{>r}$.

\subsection{Lower bound for the Newton polygon}In this section we establish a 
result analogous to \cite[Thm.~3.16]{lwx}, following the approach of \cite[\S 
6.2]{extended}. 
We now return to our general situation: $R$ is a 
Banach--Tate $\Z_p$-algebra with a multiplicative 
pseudo-uniformiser $\vp$, and $\kappa: T_{0,v} \ra R^\times$ is a 
continuous character such that the norm on $R$ is adapted to $\kappa$ and 
$\chi^v\kappa$ is trivial on $Z(K)$. 

As in \cite[(3.2.1)]{extended}, for $\alpha \in \Z_{\ge 0}$, we define 
\[n(r,\vp,\alpha) = \left\lfloor \frac{\alpha\log_p r}{\log_p|\vp|} 
\right\rfloor.\]

\begin{lemma}\label{Upestimate} Assume that there is no $x \in R$ with $1 < |x| 
< 
	|\varpi|^{-1}$.
	Let \[t = |D^\times\backslash 
	(D\otimes_F\A_{F,f})^\times/K^vI_{v,1}|\prod_{v'|p,v'\ne v}(k_{v'}-1)\] 
	
	If we define 
	\[\lambda(0) = 0, \lambda(i+1) = \lambda(i) + n(r,\varpi,\lfloor 
	i/t\rfloor)-n(r^{1/p},\varpi,\lfloor i/t\rfloor)\] then the Fredholm series 
	\[\det(1-TU_v|H^0(K^vI_{v,1},\mc{L}^v(k^v,w)\otimes_{\Zp}\mathcal{D}^{r}_{\kappa}))
	= 	\sum_{n\ge 0}c_nT^n \in R\{\{T\}\}\] satisfies $|c_n| \le 
	|\varpi|^{\lambda(n)}$.
\end{lemma}
\begin{proof}
	This is essentially \cite[Lemma 6.2.1]{extended}. Our global set-up is 
	slightly different to that in \emph{loc.~cit.}, but the proof goes through 
	verbatim.
\end{proof}

We now fix a component $\mc{W}_{\eta,\omega}$ of weight space. Let $E := 
E_\eta  \subset 
\mathbb{C}_p$ be the subfield generated by the image of $\eta$ (it is generated 
by a $p$-power root of unity). We 
fix a uniformiser $\varpi_E \in E$ 
and normalise the absolute value $|\cdot|_E$ on 
$E$ by $|p| = p^{-1}$. Let $\Lambda = 
\OO_E[[X]]$. We 
have a universal 
weight \[ \kappa_{(\eta,\omega)}: T_{0,v} \rightarrow 
\Lambda^\times\] determined by $\kappa|_{\Delta\times\Delta} = \omega$, 
$\kappa(t_v,t_v) = \eta(t_v)$ and $\kappa(\exp(p),\exp(p)^{-1}) = 1+X$.

We give the complete local ring $\Lambda$ the 
$\m_\Lambda = (\varpi_E,X)$-adic topology. Let $\mathfrak{W}_{\eta,\omega} = 
\Spa(\Lambda,\Lambda)$, 
denote its 
analytic locus by $\mathcal{W}_{\eta,\omega}^{an}$ and let $\U_1 \subset 
\mathcal{W}_{\eta,\omega}^{an}$ be 
the rational subdomain defined by \[\U_1 = \{|\varpi_E| 
\le |X| \ne 0\}.\] 
Pulling back $\U_1$ to the rigid analytic open unit disc 
$\mathcal{W}_{\eta,\omega}$ 
gives the 
`boundary annulus' $|X| \ge |\varpi_E|_E^{-1}$.

We let $R = \OO(\U_1)$. More explicitly, we can describe the elements of 
$R^\circ$ as formal 
power series \[\left\{\sum_{n \in \Z} a_n X^n: a_n \in \OO_E, 
|a_n\varpi_E^n|_E\le 
1, |a_n\varpi_E^n|_E\rightarrow 0 \hbox{ as }n\rightarrow 
-\infty\right\}.\]

$X$ is a topologically nilpotent unit in $R$ and so equipping $R$ with the norm 
$|r| = \inf\{|\varpi_E|_E^{n} \mid r \in X^nR^\circ, n \in \Z\}$ makes $R$ into 
a Banach--Tate $\Z_p$-algebra. This norm has the explicit description: 
\begin{equation}\label{Rwnorm}|\sum_{n \in \Z} a_n X^n| = 
\sup\{|a_n\varpi_E^n|_E\}.\end{equation}
Note that $X$ is a multiplicative pseudo-uniformiser and there is no $x \in R$ 
with $1 < |x| < |X|^{-1}$.

\begin{lemma}
	The norm we have defined on $R$ is adapted to $\kappa_{(\eta,\omega)}$. 
	Moreover, 
	for $t \in T_1 = (1+\varpi_v\OO_{F,v})^2$ we have 
	$|\kappa_{(\eta,\omega)}(t)-1|\le 
	|\varpi_E|_E$.
\end{lemma}
\begin{proof}If $t \in T_1$ we have $\kappa_{(\eta,\omega)}(t) - 1 = 
u\zeta(1+X)^\alpha - 1$ for some $\alpha \in \Z_p$, $u \in 
1+p\OO_{E}$ and $\zeta \in \OO_E$ a $p$-power root of unity. So $u\zeta \in 1+ 
\varpi_E\OO_E$ and
$|\kappa_{(\eta,\omega)}(t)-1|\le |\varpi_E|_E$.
\end{proof}

We can now apply Lemma \ref{Upestimate}.
\begin{corollary}\label{expUpestimate}
Consider the Fredholm series 
\[\det(1-TU_v|H^0(K^vI_{v,1},\mc{L}^v(k^v,w)\otimes_{\Zp}\mathcal{D}^{|\varpi_E|_E}_{\kappa}))
= \sum_{n\ge 0}c_nT^n \in R\{\{T\}\}.\]

Let \[t = |D^\times\backslash 
	(D\otimes_F\A_{F,f})^\times/K^vI_{v,1}|\prod_{v'|p,v'\ne v}(k_{v'}-1).\] 

\begin{enumerate}
	\item We 
	have $c_n =  \sum_{m \ge 0} b_{n,m} X^m \in \Lambda$ for all $n$.
	\item Moreover, $|b_{n,m}\varpi_E^m|_E \le |\varpi_E^{\lambda(n)}|_E$ for 
	all $m,n\ge 0$, where $\lambda(0) = 0,\lambda(1),\ldots$ is a 
	sequence of integers determined by \[\lambda(0) = 0, \lambda(i+1) = 
	\lambda(i) + \lfloor i/t\rfloor-\lfloor i/pt\rfloor.\]
    \item For $z \in \mathbb{C}_p$ with $0<v_p(z)<v_p(\varpi_E)$, we have 
    $v_p(c_n(z)) \ge \lambda(n)v_p(z)$ for every $n \ge 0$, with equality 
    holding 
    if and only if $b_{n,\lambda(n)} \in \OO_E^\times$. If 
    $b_{n,\lambda(n)} \notin \OO_E^\times$, then \[v_p(c_n(z))\ge 
    \lambda(n)v_p(z) 
    + \min\{v_p(z),v_p(\varpi_E)-v_p(z)\}.\]
\end{enumerate}
\end{corollary}
\begin{proof}
The first two parts follow from Lemma \ref{Upestimate}, exactly as in 
\cite[Thm.~6.3.2]{extended}. Note 
that $n(|\varpi_E|_E,X,\lfloor i/t \rfloor) = \lfloor i/t \rfloor$ and 
$n(|\varpi_E|_E^{1/p},X,\lfloor i/t \rfloor) =  \lfloor 1/p \lfloor 
i/t\rfloor\rfloor = \lfloor i/pt\rfloor$. 

Let $z \in \mathbb{C}_p$ with $0<v_p(z)<v_p(\varpi_E)$. It follows immediately 
from 
the second part that  
$v_p(c_n(z)) \ge \lambda(n)v_p(z)$ for every $n \ge 0$, with equality holding 
if and only if $b_{n,\lambda(n)} \in \OO_E^\times$. Finally, if 
$b_{n,\lambda(n)} \notin \OO_E^\times$ then $v_p(b_{n,\lambda(n)}) \ge 
v_p(\varpi_E)$ and the rest of the third part is easy to check.
\end{proof}

\subsection{The spectral curve and partial eigenvariety}\label{ssec:halo}We 
can now construct 
the spectral curve 
$\cZ^{U_v}(k^v,w) \rightarrow \cW$ for the compact operator $U_v$ acting on the 
spaces $H^0(K^vI_{v,1},\mc{L}^v(k^v,w)\otimes_{\Zp}\mathcal{D}^{r}_{\kappa})$, for $\ka=\ka_{U}$ ranging over the weights induced from affinoid open $U\sub \cW$, as well as 
the pullback $\cZ^{U_v}(k^v,w)_{\eta,\omega}^{> r} 
\rightarrow \cW_{\eta,\omega}^{> r}$ for $r \in (0,1)\cap\Q$ and 
$\cW_{\eta,\omega}$ a component 
 of weight space.
\begin{proposition}\label{lwx}Fix a component $\cW_{\eta,\omega}$ of weight 
space, and let 
$E = E_\eta = \Qp(\zeta_{p^c}) \subset \C_p$ be the subfield generated by the 
image of 
$\eta$ as before. Set $s_0 = 1$ if $w$ is even and $s_0 = 2$ if $w$ is 
odd\footnote{$s_0$ is the highest slope appearing in the spaces of classical 
automorphic forms with minimal weight, under the restriction that the weight 
has the same parity as $w$.}. 
Then
	\[\cZ^{U_v}(k^v,w)_{\eta,\omega}^{> |\varpi_E|} = \cZ_0\sqcup 
	\cZ_{(0,s_0)} 
	\sqcup \coprod_{\substack{k \ge 2\\k\equiv w \mod{2}}} \cZ_{k-1} 
	\sqcup 
	\cZ_{(k-1,k+1)}\] is 
	a 
	disjoint union 
	of rigid analytic 
	spaces 
	$\cZ_I$ (with $I$ denoting an interval as in the displayed formula) which 
	are 
	finite and flat over $\cW_{\eta,\omega}^{> |\varpi_E|}$. For each point $x 
	\in \cZ_I$, 
	with 
	corresponding $U_v$-eigenvalue $\lambda_x$, we have $v_p(\lambda_x) \in 
	(p-1)p^{c}v_p(z_{\kappa(x)})\cdot I$.
\end{proposition}
\begin{proof}
	This follows from Corollary \ref{expUpestimate} and Lemma 
	\ref{AtkinLehner}, as 
	in 
	\cite[Proof of Thm.~1.3]{lwx}. We sketch the argument. Note that 
	$\tilde{\eta}$ factors through 
	$(\OO_{F,v}/\varpi^{c+1})^\times$ (if $c \ge 1$, the conductor of 
	$\tilde{\eta}$ is $(\varpi^{c+1})$; if $c = 0$, $\tilde{\eta}$ is 
	either tame or trivial). By passing to a normal compact open 
	subgroup 
	of $K^v$ we may assume that \[K_{v'} \subset \begin{pmatrix}
	1+\varpi_{v'}^{c+2}\oo_{F,v'} & \oo_{F,v'} \\ \varpi_{v'}^{c+2}\oo_{F,v'} & 
	1+\varpi_{v'}^{c+2}\oo_{F,v'}
	\end{pmatrix}\] for each $v'|p$, $v' \ne v$. Now we consider points  
	$\kappa \in \cW(\C_p)$ which are locally algebraic of weight 
	$(k_v,w)$ and character $\epsilon$, with the $\epsilon_i$ factoring through 
	$(\OO_{F,v}/\varpi^{c+2})^\times$ and $\epsilon_2/\epsilon_1$ of conductor 
	$(\varpi^{c+2})$. We furthermore insist that $\kappa$ is in the component 
	$\cW_{\eta,\omega}$. This amounts to requiring that $\kappa|_{\Delta \times 
	\Delta} = \omega$ and $\tilde{\eta} = \epsilon_1\epsilon_2$. We have 
	\[v_p(z_\kappa) = v_p(\exp(p)^{k-2}(\epsilon_1/\epsilon_2)(\exp(p))-1) =  
	1/\phi(p^{c+1}) < v_p(\varpi_E) = 1/\phi(p^c)\] and hence 
	$\kappa \in \cW_{\eta,\omega}^{>|\varpi_E|}(\C_p)$.
	
	For $k \in \Z_{\ge 2}$ with $k \equiv w \pmod{2}$ we set $n_k = 
	(k-1)p^{c+1}t = \dim S(k,w,\epsilon)$, where \[t = 
	|D^\times\backslash 
	(D\otimes_F\A_{F,f})^\times/K^vI_{v,1}|\prod_{v'|p,v'\ne v}(k_{v'}-1)\] as 
	in the beginning of section \ref{ss:altrick}. We now carry out `Step I' of 
	\cite[Proof of Thm.~1.3]{lwx}: this identifies certain points on the Newton 
	polygon of the characteristic power series $\sum_{n \ge 
		0}c_n(z_\kappa)T^n$ of $U_v$ at the weight $\kappa$. We have 
	\begin{align*}v_p(z_\kappa)\lambda(n_k) &= 
	\frac{1}{\phi(p^{c+1})}\sum_{i=0}^{(k-1)p^{c+1}t-1}\left(\left\lfloor 
	\frac{i}{t} \right\rfloor - \left\lfloor \frac{i}{pt} 
	\right\rfloor\right)\\ & =  
	\frac{1}{\phi(p^{c+1})}\left(t\sum_{i=0}^{(k-1)p^{c+1}-1}i 
	-pt\sum_{i=0}^{(k-1)p^{c}-1}i\right)\\ &= \frac{(k-1)^2p^{c+1}t}{2}.
	\end{align*}
	
	So, by Lemma \ref{AtkinLehner}, $v_p(z_\kappa)\lambda(n_k)$ is equal to 
	half 
	the sum of the $U_v$-slopes 
	on $S(k,w,\epsilon)$ and $S(k,w,\epsilon^{-1})$. Combining this with Corollary 
	\ref{expUpestimate}(3) (at the weights corresponding to both $(k,w,\epsilon)$ 
	and $(k,w,\epsilon^{-1})$), we deduce that the sum of the first $n_k$ 
	$U_v$-slopes on 
	$H^0(K^vI_{v,1},\mc{L}^v(k^v,w)\otimes_{\Zp}\mathcal{D}^{|\vp_{E}|_{E}}_{\kappa})$ is 
	$\frac{(k-1)^2p^{c+1}t}{2}$ and that the Newton polygon of $\sum_{n \ge 
	0}c_n(z_\kappa)T^n$ passes through the point 
	$\mathscr{P}_k = (n_k,\lambda(n_k)v_p(z_\kappa))$.
	
	`Step II' of \cite[Proof of Thm.~1.3]{lwx} can now be carried over 
	(replacing $q$ with $p^{c+1}$), and this establishes the rest of the 
	proposition. We denote the first coordinate of the vertices of the Newton 
	polygon either side of $\mathscr{P}_k$ by $n_k^-$ and $n_k^+$ (if  
	$\mathscr{P}_k$ is itself a vertex we have $n_k^- = n_k^+ = n_k$). If 
	$n_k^- \ne n_k^+$ then the slope of the segment containing $\mathscr{P}_k$ 
	is $(k-1)$ (since this is the slope of the lower bound Newton 
	polygon) and $n_k^- \in [n_k-t,n_k]$, $n_k^+ \in [n_k,n_k+t]$. Whether or 
	not $n_k^-$ equals $n_k^+$, it now follows 
	from Corollary 
	\ref{expUpestimate}(3) that $n_k^-$ is the minimal index $i$ in 
	$[n_k-t,n_k]$ with $b_{i,\lambda(i)}\in \OO_E^\times$ and $n_k^+$ is the 
	maximal index $i$ in $[n_k,n_k+t]$ with $b_{i,\lambda(i)}\in \OO_E^\times$. 
	If we specialise to $z \in \cW_{\eta,\omega}^{> |\varpi_E|}$ we have, for 
	all $i \ge 0$: \[v_p(c_{n_k-i}(z)) \ge v_p(z)\lambda(n_k-i)\ge 
	v_p(z)(\lambda(n_k)-(k-1)\phi(p^{c+1})i)\] where the first equality is 
	strict 
	if $n_k-t \le n_k-i < n_k^-$ (minimality of $n_k^-$) and the second 
	inequality is strict if $n_k-i < n_k -t$. These strict inequalities have 
	difference between the two sides at least $\min\{v_p(z), 
	v_p(\varpi_E)-v_p(z)\}$. A similar inequality holds for 
	$n_k+i > n_k^+$. We get equalities when $n_k - i = n_k^-$ and $n_k+i = 
	n_k^+$. We 
	deduce that, if 
	$n_k^- \ne n_k^+$ then $(n_k^-, \lambda(n_k^-)v_p(z))$ and $(n_k^+, 
	\lambda(n_k^+)v_p(z))$ are consecutive vertices of the Newton polygon of 
	$\sum_{n \ge 0}c_n(z)T^n$ for all $z \in \cW_{\eta,\omega}^{> |\varpi_E|}$. 
	The slope of the segment connecting these vertices is 
	$(k-1)\phi(p^{c+1})v_p(z)$ and it therefore passes through 
	$(n_k,\lambda(n_k)v_p(z))$. If $n_k^+ = n_k^- = n_k$ then $(n_k, 
	\lambda(n_k)v_p(z))$ is a vertex of 
	the Newton polygon for all $z \in \cW_{\eta,\omega}^{> |\varpi_E|}$. Now 
	for an interval $I$ as in the statement of the proposition we 
	define $\cZ_I$ to be the (open) subspace of 
	$\cZ^{U_v}(k^v,w)_{\eta,\omega}^{> |\varpi_E|}$ given by demanding that the 
	slope lies in $(p-1)p^{c}v_p(z_{\kappa(x)})\cdot I$. Each $\cZ_I$ pulls 
	back to an affinoid open of (the pull back of) 
	$\cZ^{U_v}(k^v,w)_{\eta,\omega}^{> |\varpi_E|}$ over every affinoid open of 
	$\cW_{\eta,\omega}^{> |\varpi_E|}$ (when $I$ is an open interval we use the 
	bound in terms of 
	$\min\{v_p(z), 
	v_p(\varpi_E)-v_p(z)\}$ to see that we still get an affinoid open). By our 
	control over the Newton polygons (or Hida theory for $\cZ_0$), it follows 
	from \cite[Cor.~4.3]{bu} 
	that the $\cZ_I$ are finite 
	flat over $\cW_{\eta,\omega}^{> |\varpi_E|}$.
\end{proof}
\begin{remark}\begin{enumerate}\item In the above Proposition, we  are proving 
that the partial 
eigenvariety exhibits `halo' behaviour over a boundary annulus whose radius 
depends on the conductor of the character $\tilde{\eta}$. This dependence may 
well be an artifact of the proof.
\item The rest of the arguments in \cite{lwx} also generalise to this 
situation, showing that the slopes over a sufficiently small boundary annulus 
in weight space are a union of finitely many arithmetic progressions with the 
same common difference.
\item Another way to obtain one-dimensional families of $p$-adic automorphic 
forms for the quaternion algebra $D$ is to allow the weight to vary in the 
parallel direction. The methods described here do not seem sufficient to 
establish `halo' behaviour for these one-dimensional families --- one reason is 
that there is no 
longer a sharp numerical classicality theorem in terms of a single compact 
operator in this case.
\end{enumerate}
\end{remark}

Now we fix an integer $n \ge 1$ and assume that $K_{v'} = I'_{v',1,n-1}$ for 
each place $v'|p$ with $v' \ne v$. We set $K = K^vI_{v,1}$. Let $S$ be the set 
of finite places $w$ of 
$F$ where either $w|p$, $D_w$ is 
non-split or $D_w$ is split but $K_w \ne
\OO_{D,w}^\times$. For $w \notin S$ we have Hecke operators \[S_w = 
[K\begin{pmatrix}\varpi_w & 0\\ 0 & \varpi_w \end{pmatrix}K],~T_w = 
[K\begin{pmatrix}1 & 0\\ 0 & \varpi_w \end{pmatrix}K]\] which are independent 
of the choice of uniformiser $\varpi_w \in F_w$.

The spaces $H^0(K,\mc{L}^v(k^v,w)\otimes_{\Zp}\mathcal{D}^{r}_{\kappa})$ give 
rise to a coherent sheaf $\mc{H}$ over $\cZ^{U_v}(k^v,w)$, equipped with an 
action of 
the Hecke operators $\{U_{v'}, v'|p\}$ and $\{S_w, T_w: w \notin S\}$. 
If we 
let $\T$ denote the free commutative $\Zp$-algebra generated by
these Hecke operators, and let $\psi: \T \rightarrow \End(\mc{H})$ be the map 
induced by the Hecke action, then we have an eigenvariety 
datum $(\cW\times \mb{A}^{1},\mc{H},\T,\psi)$. Here we use the notion of eigenvariety datum as defined in \cite[\S3.1]{irreducible}, and we refer to \emph{loc. cit.} for the construction of the eigenvariety associated with an eigenvariety datum  (we remark that we could also have used the eigenvariety $(\cZ^{U_v}(k^v,w),\mc{H},\T,\psi)$, since $\mc{H}$ is supported on $\cZ^{U_v}(k^v,w)$). We denote the 
associated 
eigenvariety by $\mc{E}(k^v,w)$ and, if $r \in (0,1)$, denote its pullback to 
$\cW^{> r}$ by 
$\mc{E}(k^v,w)^{> r}$.

\begin{definition}
A \emph{classical point} of $\mc{E}(k^v,w)$ is a point with locally algebraic 
weight corresponding to a Hecke eigenvector with non-zero image under the map 
$\pi$ of (\ref{maptoclass}), whose systems of Hecke eigenvalue do not arise 
from one-dimensional automorphic representations of 
$(D\otimes_F\A_{F})^\times$. 
\end{definition}
The points excluded in the above definition do not correspond to 
classical 
Hilbert modular forms under the Jacquet--Langlands correspondence. They all 
have parallel weight $2$, and their 
$U_{v'}$ eigenvalues $\alpha_{v'}$ satisfy $v_p(\alpha_{v'}) = 1$ for all 
$v'|p$. 
\begin{proposition}\label{prop:wtkpoints}
	\leavevmode
	\begin{enumerate}
	\item $\mc{E}(k^v,w)$ is equidimensional of dimension 
	$1$ and flat over $\cW$.
	\item Let 
	$C$ be an 
	irreducible component of $\mc{E}(k^v,w)$ and let $k \ge 2$ be an integer 
	with the same parity as 
	$w$. Then $C$ contains a point which is locally algebraic of weight $(k,w)$ 
	and $U_v$-slope $< k-1$ (in particular, this point is classical).
	\end{enumerate}
\end{proposition} 
\begin{proof}
	\begin{enumerate}
	\item This part is proved in the same way as the well-known analogous 
	statement for the Coleman--Mazur eigencurve (for example, see 
	\cite[\S6.1]{extended}).
	\item This is a consequence of Proposition \ref{lwx}. Indeed, $C$ maps to a 
	single component $\cW_{\eta,\omega}$ of weight 
	space, and we let $E = \Qp(\zeta^c)$ be as in that proposition. Then any 
	irreducible 
	component $C'$ of $C^{> |\varpi_E|}$ is finite flat over an irreducible 
	component of $\cZ^{U_v}(k^v,w)^{> |\varpi_E|}$.  It follows from 
	Proposition \ref{lwx} that $C'$ is finite flat over $\cW_{\eta,\omega}^{> 
	|\varpi_E|}$ 
	and there is an interval $I$ such that for each point 
	$x \in C'$, with 
	corresponding $U_v$-eigenvalue $\lambda_x$, we have $v_p(\lambda_x) \in 
	(p-1)p^{c}v_p(z_{\kappa(x)})\cdot I$. If $v_p(z_{\kappa(x)})$ is 
	sufficiently small, then $v_p(\lambda_x) < k-1$. To finish the proof, we 
	note 
	that we can 
	find a locally algebraic point of weight $(k,w)$ in $\cW_{\eta,\omega}^{> 
	|\varpi_E|}$ with $v_p(z_{\kappa(x)})$ as small as we like, by choosing a 
	character $\epsilon = (\epsilon_1,\epsilon_2)$ 
	(compatible with $\eta$ and $\omega$)
	such that  $\epsilon_1(\exp(p))/\epsilon_2(\exp(p))$ is a $p$-power root of 
	unity of sufficiently large order.\qedhere 
\end{enumerate}
\end{proof}

\begin{remark}
	The existence of partial eigenvarieties (where weights vary above a proper 
	subset of the places dividing $p$) is probably well-known to experts, but 
	they have not been utilised so much in the literature. They were discussed 
	(for 
	definite unitary groups) and 
	used in works of Chenevier 
	\cite{cheapp} and Chenevier--Harris \cite{chehar} to remove regularity 
	hypotheses from theorems concerning the existence and local--global 
	compatibility of Galois representations associated to automorphic 
	representations. They can be viewed as eigenvarieties defined with respect 
	to certain non-minimal parabolic subgroups of the relevant group (see 
	\cite{loe}) 
	--- to recover the setting of this article we let $P \subset 
	(\Res_{F/\Q}\GL_2)_{\Qp} = 
	\prod_{v'|p}\Res_{F_{v'}/\Q_p}\GL_2$ be the parabolic with component 
	labelled by our fixed place $v$ equal to the upper triangular Borel and 
	other 
	components equal to the whole group $\GL_2$. See, for example, \cite{ding} 
	for 
	another application of eigenvarieties defined with respect to a non-minimal 
	parabolic.  We also 
	note that partial 
	Hilbert modular eigenvarieties have been 
	used recently by Barrera, Dimitrov and Jorza to study the 
	exceptional zero conjecture for Hilbert modular forms \cite{BDJorza}.
\end{remark}

\section{The parity conjecture}\label{sec:parity}

We can now apply the preceding results to establish some new cases of the 
parity conjecture for Hilbert modular forms, using \cite[Thms.~A and B]{px}. 
First we need to discuss the family of Galois representations carried by the 
partial eigenvarieties and their $p$-adic Hodge theoretic properties.

\subsection{Galois representations}\label{ssec:galrep}
We begin by noting that there is a continuous $2$-dimensional pseudocharacter 
$T: G_F \rightarrow 
\OO_{\mc{E}(k^v,w)}$ with $T(\Frob_w) = T_w$ for all $w \notin S$. This 
follows 
from \cite[Prop.~7.1.1]{chegln}, the Zariski density of classical points in 
$\mc{E}(k^v,w)$ and the existence of Galois representations associated to 
Hilbert modular forms.

We denote the normalization of $\mc{E}(k^v,w)$ by $\tilde{\mc{E}}$, and we 
also denote by $T$ the pseudocharacter on $\tilde{\mc{E}}$ obtained by 
pullback from $\mc{E}(k^v,w)$. We denote by  $\tilde{\mc{E}}^\mathrm{irr} 
\subset \tilde{\mc{E}}$ the 
(Zariski open, see \cite[Example 2.20]{che}) locus where $T$ is absolutely 
irreducible. Note also that $\tilde{\mc{E}}^\mathrm{irr}$ is Zariski dense in $\tilde{\mc{E}}$, since classical points have irreducible Galois representations.

\begin{proposition}\label{liftdet}
	There is a locally free rank $2$ 
	$\OO_{\tilde{\mc{E}}^\mathrm{irr}}$-module $V$ equipped 
	with a continuous representation \[\rho: G_{F,S} \rightarrow 
	\mathrm{GL}_{\OO_{\tilde{\mc{E}}^\mathrm{irr}}}(V)\] satisfying 
	\[\det(X-\rho(\Frob_w)) = 
	X^2 - T_w X +q_wS_w\] for all $w \notin S$ (where $\Frob_w$ denotes an 
	arithmetic 
	Frobenius).
\end{proposition}
\begin{proof}
	First we recall that $T$ canonically lifts to a continuous representation 
	$\rho: G_{F,S} \rightarrow \mc{A}^\times$ where $\mc{A}$ is an Azumaya 
	algebra 
	of rank $4$ over $\tilde{\mc{E}}^\mathrm{irr}$, by 
	\cite[Cor.~7.2.6]{chegln} 
	(see also 
	\cite[Prop.~G]{che}). It remains to show that $\mc{A}$ is isomorphic to the 
	endomorphism algebra of a vector bundle.
	
	On the other hand, \cite[Thm. 5.1.2]{cm} and the remark 
	appearing 
	after this theorem shows that there is an admissible cover of 
	$\tilde{\mc{E}}^\mathrm{irr}$ by affinoid opens $\{U_i\}_{i \in I}$, 
	together with continuous representations \[\rho_i: G_{F,S} \rightarrow 
	\GL_2(\oo(U_i))\] lifting the pseudocharacters $T|_{U_i}$.
	
	By the uniqueness part of \cite[Lem.~7.2.4]{chegln}, we deduce that for 
	each 
	affinoid open $U_i$, we have an isomorphism of $\oo(U_i)$-algebras 
	$\mc{A}(U_i) 
	\cong M_2(\oo(U_i))$. Now, by the standard argument relating the Brauer 
	group to cohomology (see \cite[Thm.~IV.2.5]{milne}; the constructions using 
	either \v{C}ech cohomology or gerbes can be applied, since \v{C}ech 
	cohomology 
	coincides with usual cohomology on 	$\tilde{\mc{E}}^\mathrm{irr}$ by 
	\cite[Prop.~1.4.4]{vdp}) we 
	can associate to $\mc{A}$ an element $F_{\mc{A}}$ of 
	$H^2(\tilde{\mc{E}}^\mathrm{irr},\oo_{\tilde{\mc{E}}^\mathrm{irr}}^\times)$ 
	(crucially, this is the cohomology on the rigid analytic site, not 
	\'{e}tale cohomology). This element is trivial if and only if $\mc{A}$ is 
	isomorphic to the 
	endomorphism algebra of a vector bundle.
	
	Since $\tilde{\mc{E}}^\mathrm{irr}$ is separated and one-dimensional, 
	$H^2(\tilde{\mc{E}}^\mathrm{irr},\mc{F})$ vanishes for any Abelian sheaf 
	$\mc{F}$ by \cite[Cor.~2.5.10, Rem.~2.5.11]{djvdp}. In particular, 
	$F_{\mc{A}}$ 
	is trivial and we are done.
\end{proof}
\begin{remark}
	In fact, \cite[Thm. 5.1.2]{cm} applies over the whole of $\tilde{\mc{E}}$, 
	not just the irreducible locus, so there are representations $\rho_i:  
	\oo(U_i)[G_{F,S}] \rightarrow M_2(\oo(U_i))$ lifting $T$ 
	over each member of an admissible affinoid cover $\{U_i\}_{i \in I}$ of 
	$\tilde{\mc{E}}$. 
	Shrinking the cover if necessary, we may assume that the intersection of 
	any two distinct covering affinoids $U_i\cap U_j$ is contained in 
	$\tilde{\mc{E}}^\mathrm{irr}$, so we obtain a canonical isomorphism of 
	$\oo_{U_i\cap U_j}$-algebras $M_2(\oo_{U_i\cap U_j}) \cong M_2(\oo_{U_i\cap 
	U_j})$ intertwining the representations $\rho_i$ and $\rho_j$.

	This gives the gluing data necessary to define a Galois 
	representation to an Azumaya algebra, which as above is isomorphic to the 
	endomorphism algebra of a vector bundle (by the vanishing of 
	$H^2(\tilde{\mc{E}},\oo_{\tilde{\mc{E}}}^\times)$). So we finally obtain a 
	(possibly non-canonical) Galois representation on a vector bundle over 
	$\tilde{\mc{E}}$, although 
	in what 
	follows we will just work over the irreducible locus as this suffices for 
	our purposes.
\end{remark}

If $z$ is a point of $\tilde{\mc{E}}^\mathrm{irr}$ we denote by $V_z$ the 
specialisation of 
$V$ at $z$, which we regard as a $2$-dimensional representation of $G_{F}$ over 
the residue field $k(z)$. Note that we obtain a continuous character $\det 
\rho: G_F 
\rightarrow 
\Gamma(\tilde{\mc{E}}^\mathrm{irr},\OO_{\tilde{\mc{E}}^\mathrm{irr}})^\times$  
with $\det \rho(z) = \det 
V_z$. In fact $\det \rho$ is constant on connected components of 
$\tilde{\mc{E}}^\mathrm{irr}$. Indeed, for every classical point $z$, $\det 
\rho(z)$ is Hodge--Tate of weight $-1-w$ at every place dividing $p$, and the 
results cited in the proof of Proposition \ref{lgcvprime} below show that this 
property extends to all points $z$. Class field theory and the fact that 
Hodge--Tate 
characters are locally algebraic \cite[Thm.~3, Appendix to Ch.~III]{serre} 
shows that each $\det\rho(z)$ is the product of a finite order character and 
the $(1+w)$-power of the cyclotomic character, so $\det\rho$ is constant on 
connected components.

\begin{proposition}\label{lgcv}
	Let $z$ be a classical point of $\tilde{\mc{E}}^\mathrm{irr}$, locally 
	algebraic of weight 
	$(k,w)$ and character $\epsilon$, and with $U_v$-eigenvalue $\alpha$. Set 
	$\wt{\alpha} = \alpha p^{1+\frac{w-k}{2}}$. Then 
	$V_z|_{G_{F_v}}$ is potentially semistable with Hodge--Tate 
	weights\footnote{We 
		use covariant Dieudonn\'{e} modules so the cyclotomic character has 
		Hodge--Tate 
		weight $-1$.} 
	$(-\frac{w+k}{2},-(1+\frac{w-k}{2}))$, and the associated 
	(Frobenius-semisimplified) Weil--Deligne 
	representation is of the form 
	\[WD(V_z|_{G_{F_v}})^{F-ss} = \left(ur(\wt{\beta})^{-1}\epsilon_1\oplus 
	ur(\wt{\alpha})^{-1}\epsilon_2,N\right)\] for some $\wt{\beta}\in\C_p$ 
	(determined by the determinant of $V_z$) which satisfies 
	$v_p(\wt{\alpha})+v_p(\wt{\beta}) = 1+w$. $N$ is non-zero if and only if 
	$\epsilon_1 = \epsilon_2$ and $\wt{\alpha}/\wt{\beta} = p^{- 1}$. Here we 
	are  
	interpreting characters of $F_v^\times$ as characters of $W_{F_v}$ via the 
	Artin reciprocity map, normalised to 
	take a uniformiser to a geometric Frobenius. 
\end{proposition}
\begin{proof}
	This follows from the local-global compatibility theorem of Saito 
	\cite{saito}, as completed by Skinner \cite{sklgc} and T.~Liu \cite{tliu}. 
	Note that $\alpha$ is an eigenvalue for the $\star$-action of $U_v$ on 
	$S(k,w,\epsilon)$ so $\wt{\alpha}$ is the corresponding eigenvalue for the 
	standard 
	action of $U_v$.
\end{proof}
\begin{definition}
	We say that a classical point $z$ of weight $(k,w)$ is \emph{critical} if 
	$V_z|_{G_{F_v}}$ is decomposable and $v_p(U_v(z)) = k-1$. 
	
\end{definition}

\begin{remark}
	It follows from weak admissibility that $WD(V_z|_{G_{F_v}})$ fails to be 
	Frobenius-semisimple if and only if the two characters 
	$ur(\wt{\beta})^{-1}\epsilon_1, ur(\wt{\alpha})^{-1}\epsilon_2$ are equal 
	(in which case we also have $N = 0$).
\end{remark}

In what remains of this section we give details about the triangulations of the 
$p$-adic Galois representations for noncritical classical points of 
$\tilde{\mc{E}}^\mathrm{irr}$ and apply the results of \cite{kpx} to establish 
the existence of global triangulations for the family of Galois representations 
over this eigenvariety. Everything here should be unsurprising to experts, but 
we have written out some details because in the literature it is common to 
restrict to the semistable case, whilst it is crucial for us to consider 
classical points which are only potentially semistable.

\subsubsection{Triangulation at $v$ for classical points}
Suppose $z \in \tilde{\mc{E}}^\mathrm{irr}(L)$ is a noncritical classical point 
(with $L/\Qp$ 
finite), locally algebraic of weight 
$(k,w)$ and character $\epsilon$, and with $U_v$-eigenvalue $\alpha$. Suppose 
moreover that $V_z|_{G_{F_v}}$ is indecomposable and potentially crystalline. 
It follows from Proposition \ref{lgcv} that this representation becomes 
crystalline over an Abelian extension of $F_v$. So the 
$L$-Weil--Deligne representation $WD(V_z|_{G_{F_v}})$ has an admissible 
filtration after extending scalars to the $\Qp$-algebra $L_\infty = 
\cup_N 
L\otimes_{\Qp}\Qp(\zeta_{p^N})$ (see \cite[\S 4.4]{colmeztri}). Moreover, we 
can explicitly describe this admissible filtered Weil--Deligne module $D_z$ as 
in \cite[\S 4.5]{colmeztri}. Indeed, when the action of $W_{F_v}$ on $D_z$ is 
semisimple, we let $e_1$ and $e_2$ be basis vectors 
with $W_{F_v}$ action given by $ur(\wt{\beta})^{-1}\epsilon_1$ and 
$ur(\wt{\alpha})^{-1}\epsilon_2$ respectively. When the action of $W_{F_v}$ is 
not semisimple, we have $ur(\wt{\beta})^{-1}\epsilon_1 = 
ur(\wt{\alpha})^{-1}\epsilon_2$ and there are basis vectors $e_1, e_2$ with 
$W_{F_v}$ action given by $\sigma e_1 = 
ur(\wt{\beta})^{-1}\epsilon_1(\sigma)e_1$ and $\sigma e_2 = 
ur(\wt{\beta})^{-1}\epsilon_1(\sigma)(e_2 - 
(\deg(\sigma))e_1)$\footnote{$\sigma$ maps 
to the $\deg(\sigma)$ power of arithmetic Frobenius in the absolute Galois 
group of the residue field.}.   The 
filtration is given by 

\begin{equation*}
\Fil^i(L_\infty\otimes_L D_z)=
\begin{cases}
0 & \text{if}\ i > -(1+\frac{w-k}{2}) \\
L_\infty\cdot(\gamma e_1 + e_2) & \text{if}\ -\frac{w+k}{2} < i \le 
-(1+\frac{w-k}{2})\\
L_\infty\otimes_L D_z & \text{if}\ i\le -\frac{w+k}{2}
\end{cases}
\end{equation*}
where $\gamma \in L_\infty$ is the (invertible) value of a certain Gauss sum in 
the semisimple case and $\gamma = 0$ in the non-semisimple case.

It follows from \cite[Prop.~4.13]{colmeztri} that $V_z|_{G_{F_v}}$ is 
trianguline. To describe the triangulations we adopt the notation of 
\cite{colmeztri}. So $\mathscr{R}$ denotes the Robba ring over $L$ and if 
$\delta: \Qp^\times \rightarrow L^\times$ is a continuous character 
$\mathscr{R}(\delta)$ denotes that $(\phi,\Gamma)$-module obtained by 
multiplying the action of $\phi$ on $\mathscr{R}$ by $\delta(p)$ and the action 
of $\gamma \in \Gamma = \Gal(\Qp(\mu_{p^\infty})/\Qp)$ by 
$\delta(\chi_{\mathrm{cyc}}(\gamma))$. If $V$ is an $L$-representation of 
$G_{\Qp}$ we denote by 
$\mathbf{D}_{\text{rig}}(V)$ the slope zero $(\phi,\Gamma)$-module over 
$\mathscr{R}$ associated to $V$ by \cite[Prop.~1.7]{colmeztri}.

Now we have recalled the necessary notation we can recall the precise statement 
of \cite[Prop.~4.13]{colmeztri}: we have two triangulations (they coincide in 
the non-semisimple case) \begin{align*} 0 
&\rightarrow  
\mathscr{R}(x^{\frac{w+k}{2}}ur(\wt{\beta})^{-1}\epsilon_1) \rightarrow 
\mathbf{D}_{\text{rig}}(V_z) \rightarrow 
\mathscr{R}(x^{(1+\frac{w-k}{2})}ur(\wt{\alpha})^{-1}\epsilon_2)\rightarrow 0 \\
0 &\rightarrow  
\mathscr{R}(x^{\frac{w+k}{2}}ur(\wt{\alpha})^{-1}\epsilon_2) \rightarrow 
\mathbf{D}_{\text{rig}}(V_z) \rightarrow 
\mathscr{R}(x^{(1+\frac{w-k}{2})}ur(\wt{\beta})^{-1}\epsilon_1)\rightarrow 0
\end{align*}

We can rewrite the first of these triangulations as

\[  0 \rightarrow  
\mathscr{R}(\delta^{-1}\det \rho (z)) \rightarrow 
\mathbf{D}_{\text{rig}}(V_z) \rightarrow 
\mathscr{R}(\delta(z))\rightarrow 0 \]

where we define the continuous character 
$\delta: F_v^\times \rightarrow 
\Gamma(\tilde{\mc{E}}^\mathrm{irr},\OO_{\tilde{\mc{E}}^\mathrm{irr}})^\times$ 
by $\delta(p) = U_v^{-1}$ 
and 
$\delta|_{\OO_{F_v}^\times} = \kappa|_{\{1\} \times \OO_{F_v}^\times}$.

Next, we consider the case where $V_z|_{G_{F_v}}$ is not potentially 
crystalline. 
Following \cite[\S 4.6]{colmeztri} we can describe the associated admissible 
filtered Weil--Deligne module $D_z$: letting $e_2$ and $e_1$ be basis vectors 
with $W_{F_v}$ action given by $ur(p\wt{\alpha})^{-1}\epsilon_2$ and 
$ur(\wt{\alpha})^{-1}\epsilon_2$ respectively, we have $Ne_1 = e_2$, $Ne_2 = 
0$, and the filtration is given by 

\begin{equation*}
\Fil^i(L_\infty\otimes_L D_z)=
\begin{cases}
0 & \text{if}\ i > -(1+\frac{w-k}{2}) \\
L_\infty\cdot(e_1 - \mathscr{L} e_2) & \text{if}\ -\frac{w+k}{2} < i \le 
-(1+\frac{w-k}{2})\\
L_\infty\otimes_L D_z & \text{if}\ i\le -\frac{w+k}{2}
\end{cases}
\end{equation*}
for some $\mathscr{L} \in L$. 

It follows from \cite[Prop.~4.18]{colmeztri} that $V_z|_{G_{F_v}}$ is 
trianguline. Indeed, we have a triangulation \[ 0 \rightarrow  
\mathscr{R}(x^{\frac{w+k}{2}}|x|ur(\wt{\alpha})^{-1}\epsilon_2) \rightarrow 
\mathbf{D}_{\text{rig}}(V_z) \rightarrow 
\mathscr{R}(x^{(1+\frac{w-k}{2})}ur(\wt{\alpha})^{-1}\epsilon_2)\rightarrow 0 \]

which again can be rewritten as

\[  0 \rightarrow  
\mathscr{R}(\delta^{-1}\det \rho (z)) \rightarrow 
\mathbf{D}_{\text{rig}}(V_z) \rightarrow 
\mathscr{R}(\delta(z))\rightarrow 0.\]

The final case we have to consider is when $V_z|_{G_{F_v}}$ is decomposable. 
Then we have $v_p(\alpha) = 0$ (by the noncritical assumption on $z$) and 
\[\mathbf{D}_{\text{rig}}(V_z) = 
\mathscr{R}(x^{\frac{w+k}{2}}ur(\wt{\beta})^{-1}\epsilon_1) \oplus  
\mathscr{R}(x^{(1+\frac{w-k}{2})}ur(\wt{\alpha})^{-1}\epsilon_2)\]

We can summarise our discussion in the following:
\begin{proposition}\label{strictri}
	Suppose $z \in \tilde{\mc{E}}^\mathrm{irr}$ is a noncritical classical 
	point. 
	Then 
	there is a 
	triangulation \[  0 \rightarrow  
	\mathscr{R}(\delta^{-1}\det \rho (z)) \rightarrow 
	\mathbf{D}_{\text{rig}}(V_z) \rightarrow 
	\mathscr{R}(\delta(z))\rightarrow 0\] where we define the continuous 
	character 
	$\delta: F_v^\times \rightarrow 
	\Gamma(\tilde{\mc{E}}^\mathrm{irr},\OO_{\tilde{\mc{E}}^\mathrm{irr}})^\times$
	by $\delta(p) = U_v^{-1}$ 
	and 
	$\delta|_{\OO_{F_v}^\times} = \kappa|_{\{1\} \times \OO_{F_v}^\times}$.	
	Moreover, this triangulation is \emph{strict} in the sense of 
	\cite[Def.~6.3.1]{kpx}
\end{proposition}
\begin{proof}
	Let $D = 
	\mathbf{D}_{\text{rig}}(V_z)\otimes_{\mathscr{R}}\mathscr{R}(\delta^{-1}\det\rho
	)^{-1}(z)$. The only thing remaining to be verified is that the 
	triangulation is 
	strict, which comes down to checking that 
	$D^{\phi = 1, \Gamma = 1}$ is one-dimensional over 
	the 
	coefficient field $L$. Suppose for a contradiction that this space has 
	dimension $2$. Then, by \cite[Prop.~2.1]{colmeztri} we have 
	$(\delta^{-1}\det \rho 
	)(z) = x^i\delta(z)$ for some $i \in \Z_{\ge 0}$. Comparing 
	these characters on the intersection of the kernels of $\epsilon_1$ and 
	$\epsilon_2$ we deduce that $i = k-1$. We have $(\delta^{-1}\det \rho 
	)(z) = x^{k-1}\delta(z)$ if 
	and only if the characters  $ur(\wt{\beta})^{-1}\epsilon_1$, 
	$ur(\wt{\alpha})^{-1}\epsilon_2$ appearing in $D_z$ are equal. It follows 
	from the 
	correspondence between filtered Weil--Deligne modules and potentially 
	semistable $(\phi,\Gamma)$-modules established by \cite[Thm.~A]{berger} 
	that 
	$D[\frac{1}{t}]^{\Gamma = 1}$ is two-dimensional with non-trivial unipotent 
	$\phi$ action, so  $D[\frac{1}{t}]^{\phi = 1, \Gamma = 1}$ is 
	one-dimensional which gives the desired statement.
\end{proof}

We can now apply the results of \cite{kpx} to establish the existence of a 
global triangulation for the family of Galois representations $V$. In the below 
statement we adopt the notation of \cite[6.2.1, 6.2.2]{kpx} so for a rigid 
space $X/\Qp$ and a character $\delta_X : \Qp^\times \rightarrow 
\Gamma(X,\oo_X)^\times$ we have a free rank one $(\phi,\Gamma)$-module 
$\mathscr{R}_X(\delta_X)$ over the sheaf of Robba rings $\mathscr{R}_X$. We 
also have a rank two $(\phi,\Gamma)$-module 
$\mathbf{D}_{\mathrm{rig}}(V|_{G_{F_v}})$ over 
$\mathscr{R}_{\tilde{\mc{E}}^\mathrm{irr}}$.

\begin{corollary}\label{vtrifam}
	\leavevmode
	\begin{enumerate}
		\item $\mathbf{D}_{\mathrm{rig}}(V|_{G_{F_v}})$ is densely 
		pointwise 
		strictly trianguline in 
		the sense of \cite[Def.~6.3.2]{kpx}, with respect to the ordered 
		parameters 
		$\delta^{-1}\det \rho , \delta$ and the Zariski dense subset of 
		noncritical 
		classical 
		points.
		\item There are line bundles $\mathscr{L}_1$ and $\mathscr{L}_2$ over 
		$\tilde{\mc{E}}^\mathrm{irr}$ and an exact sequence \[0 \rightarrow 
		\mathscr{R}_{\tilde{\mc{E}}^\mathrm{irr}}(\delta^{-1}\det \rho 
		)\otimes_{\oo_{\tilde{\mc{E}}^\mathrm{irr}}}\mathscr{L}_1 
		\rightarrow 
		\mathbf{D}_{\mathrm{rig}}(V|_{G_{F_v}}) \rightarrow 
		\mathscr{R}_{\tilde{\mc{E}}^\mathrm{irr}}(\delta)\otimes_{\oo_{\tilde{\mc{E}}^\mathrm{irr}}}\mathscr{L}_2\]
		
		with the 
		cokernel of the final map vanishing over a Zariski open subset which 
		contains 
		every noncritical classical point. In particular, this sequence induces 
		a 
		triangulation of $\mathbf{D}_{\mathrm{rig}}(V_z|_{G_{F_v}})$ at every 
		noncritical classical point.
	\end{enumerate}
\end{corollary}
\begin{proof}
	The first assertion follows immediately from Proposition \ref{strictri}. 
	The existence of the global triangulation and the fact that the cokernel of 
	the final map vanishes over a Zariski open subset containing every 
	noncritical 
	classical point follows from 
	\cite[Cor.~6.3.10]{kpx}. To apply this Corollary we have to check that 
	there exists an admissible affinoid 
	cover of $\tilde{\mc{E}}^\mathrm{irr}$ such that the noncritical classical 
	points are Zariski dense in each member of the cover --- this property is 
	the definition of `Zariski dense' in \cite[Def.~6.3.2]{kpx}. We prefer to 
	reserve the terminology Zariski dense for the standard property that a set 
	of points is not contained in a proper analytic subset. The (a priori) 
	stronger density statement needed to apply \cite[Cor.~6.3.10]{kpx} follows 
	from \cite[Lem.~5.9]{chj} (it can be shown that 
	$\tilde{\mc{E}}^\mathrm{irr}$ 
	has an increasing cover by affinoids as in \cite[Rem.~5.10]{chj}).
\end{proof}

\subsubsection{$p$-adic Hodge theoretic properties at $v' \ne v$}
Now we let $v'|p$ be a place of $F$ with $v' \ne v$.
\begin{proposition}\label{lgcvprime}
	\leavevmode
	\begin{enumerate}
		\item Let $z \in \tilde{\mc{E}}^\mathrm{irr}$. Then $V_z|_{G_{F_{v'}}}$ 
		is 
		potentially 
		semistable 
		with 
		Hodge--Tate weights $-\frac{w+k_{v'}}{2}$ and $-(1+\frac{w-k_{v'}}{2})$.
		
		\item The restriction to $I_{F_{v'}}$ of the Weil--Deligne 
		representation 
		$WD(V_z|_{G_{F_{v'}}})$ depends only on the connected component of $z$ 
		in 
		$\tilde{\mc{E}}^\mathrm{irr}$.
		
		\item If $z$ is a classical point with $U_{v'}(z) = 
		\alpha \ne 0$, and we set $\wt{\alpha} = \alpha 
		p^{1+\frac{w-k_{v'}}{2}}$ then 
		the associated 
		(Frobenius-semisimplified) Weil--Deligne 
		representation is of the form 
		\[WD(V_z|_{G_{F_{v'}}})^{F-ss} = \left(ur(\wt{\beta})^{-1}\psi_1\oplus 
		ur(\wt{\alpha})^{-1}\psi_2,N\right)\] for some characters $\psi_i$ of 
		$\OO_{F,v'}^\times$ and $\wt{\beta}\in\C_p$ 
		(determined by the determinant of $V_z$) satisfying 
		$v_p(\wt{\alpha})+v_p(\wt{\beta}) = 1+w$. $N$ is non-zero if and only 
		if 
		$\psi_1 = \psi_2$ and $\wt{\alpha}/\wt{\beta} = p^{- 1}$.
	\end{enumerate}
\end{proposition}
\begin{proof}
	For classical points, the first and third assertions follow from 
	local--global compatibility, as in Proposition \ref{lgcv}. Note that for 
	the third part, a calculation similar to Lemma \ref{lem:fslevelindep} shows 
	that we can assume that the integer $n$ defining the level at the place 
	$v'$ is minimal such that the relevant space of invariants 
	$\pi_{v'}^{I'_{v',1,n-1}}$ is non-zero. The rest of the 
	Proposition then follows from \cite[Thm.~B, Thm.~C]{beco}, which were 
	generalised to reduced quasicompact (and quasiseparated) rigid spaces in 
	\cite[Cor.~3.19]{cheapp}, and \cite[Lem.~7.5.12]{bc} (which allows us to 
	identify the Hodge--Tate weights precisely). Note that every point 
	of $\tilde{\mc{E}}^\mathrm{irr}$ is 
	contained in a quasicompact open subspace with a Zariski dense set of 
	classical points (as in the proof of Corollary \ref{vtrifam}, see 
	\cite[Lem.~5.9, Rem.~5.10]{chj}).
\end{proof}
\begin{definition}\label{vprimefs}
	We denote by $\tilde{\mc{E}}^{v'-\mathrm{fs}}$ the Zariski open subspace of 
	$\tilde{\mc{E}}^{\mathrm{irr}}$ where $U_{v'}$ is non-zero.
\end{definition}
\begin{lemma}\label{wdrepvprime}
	\leavevmode
	\begin{enumerate}
		\item If $z \in \tilde{\mc{E}}^{v'-\mathrm{fs}}$, with $U_{v'}(z) = 
		\alpha \ne 0$, and we set $\wt{\alpha} = \alpha 
		p^{1+\frac{w-k_{v'}}{2}}$ then 
		the associated 
		(Frobenius-semisimple) Weil--Deligne 
		representation is of the form 
		\[WD(V_z|_{G_{F_{v'}}})^{F-ss} = \left(ur(\wt{\beta})^{-1}\psi_1\oplus 
		ur(\wt{\alpha})^{-1}\psi_2,N\right)\] for some characters $\psi_i$ of 
		$\OO_{F,v'}^\times$ and $\wt{\beta}\in\C_p$ 
		(determined by the determinant of $V_z$) satisfying 
		$v_p(\wt{\alpha})+v_p(\wt{\beta}) = 1+w$.
		
		\item $\tilde{\mc{E}}^{v'-\mathrm{fs}}$ is a union of connected 
		components of 
		$\tilde{\mc{E}}^{\mathrm{irr}}$.
	\end{enumerate}
\end{lemma}
\begin{proof}
	We begin by establishing the first part. Suppose $z \in 
	\tilde{\mc{E}}^{v'-\mathrm{fs}}$. It follows from the second and 
	third parts of Proposition \ref{lgcvprime} that the Weil--Deligne 
	representation at $z$ has a reducible Weil group representation.  Since the 
	Weil--Deligne representation varies analytically over 
	$\tilde{\mc{E}}^{\mathrm{irr}}$ \cite[Thm.~C]{beco}, we 
	can 
	consider the characteristic polynomial of a lift of $\Frob_{v'}$ over 
	$\tilde{\mc{E}}^{\mathrm{irr}}$. The third part of Proposition 
	\ref{lgcvprime} and Zariski 
	density of classical points implies that 
	$p^{1+\frac{w-k_{v'}}{2}}U_{v'}$ is a root of this polynomial over 
	$\tilde{\mc{E}}^{v'-\mathrm{fs}}$, and this proves that the Weil--Deligne 
	representation has the desired form over all of 
	$\tilde{\mc{E}}^{v'-\mathrm{fs}}$.
	
	For the second part, we suppose that $z \in \tilde{\mc{E}}^{\mathrm{irr}}$ 
	is 
	an element of 
	the 
	Zariski closure of $\tilde{\mc{E}}^{v'-\mathrm{fs}}$. We need to show that 
	$U_{v'}(z) \ne 0$. The same argument we used to establish the first part 
	shows that the Weil--Deligne 
	representation at $z$ has a reducible Weil group representation with a
	$\Frob_{v'}$ eigenvalue given by $p^{1+\frac{w-k_{v'}}{2}}U_{v'}(z)$. It 
	follows immediately that $U_{v'}(z) \ne 0$.		
\end{proof}

\begin{lemma}\label{lem:redlocus}
	The locus of points $z \in \tilde{\mc{E}}^{v'-\mathrm{fs}}$ with 
	$V_z|_{G_{F_{v'}}}$ reducible is equal to the locus of points with 
	$v_p(U_{v'}(z)) = 0$ or $k_{v'}-1$.
	
	Moreover, this locus is a union of connected components of 
	$\tilde{\mc{E}}^{v'-\mathrm{fs}}$.
\end{lemma}
\begin{proof}
	It follows from Lemma \ref{wdrepvprime} and the explicit description of 
	admissible filtered Weil--Deligne modules 
	\cite[\S4.5]{colmeztri} that $V_z|_{G_{F_{v'}}}$ is reducible if and only 
	if $v_p(U_{v'}(z)) = 0$ or $k_{v'}-1$. Since reducibility of a 
	representation is a Zariski closed condition, and the conditions on the 
	slope are open, we deduce that the reducible locus is a union of connected 
	components.
\end{proof}
\begin{definition}We denote the locus of $z \in 
	\tilde{\mc{E}}^{v'-\mathrm{fs}}$ 
	with $V_z|_{G_{F_{v'}}}$ irreducible or reducible with $v_p(U_{v'}(z)) = 0$ 
	by 
	$\tilde{\mc{E}}^{v'-\mathrm{fs},\mathrm{good}}$ (the preceding lemma 
	implies 
	that 
	this is a union of connected components of 
	$\tilde{\mc{E}}^{v'-\mathrm{fs}}$).
\end{definition}

\begin{corollary}\label{trivprime}
	There are line bundles $\mathscr{L}_1$ and $\mathscr{L}_2$ over 
	$\tilde{\mc{E}}^{v'-\mathrm{fs},\mathrm{good}}$, a continuous character 
	$\delta: 
	F_{v'}^\times \rightarrow 
	\Gamma(\tilde{\mc{E}}^{v'-\mathrm{fs},\mathrm{good}},\OO_{\tilde{\mc{E}}})^\times$
	with $\delta(p) = U_{v'}^{-1}$ 
	and a short exact sequence \[0 
	\rightarrow 
	\mathscr{R}_{\tilde{\mc{E}}^{v'-\mathrm{fs},\mathrm{good}}}(\delta^{-1}\det 
	\rho 
	)\otimes\mathscr{L}_1 \rightarrow 
	\mathbf{D}_{\mathrm{rig}}(V_{\tilde{\mc{E}}^{v'-\mathrm{fs},\mathrm{good}}}|_{G_{F_{v^{\prime}}}})
	\rightarrow 
	\mathscr{R}_{\tilde{\mc{E}}^{v'-\mathrm{fs},\mathrm{good}}}(\delta)\otimes\mathscr{L}_2
	\rightarrow 0.\]
\end{corollary}
\begin{proof}
	This is proved exactly as Corollary \ref{vtrifam}. 
\end{proof}
Finally, we note that if $\delta: \Qp^\times \to L^\times$ is a continuous 
character such that there exists  $i \in \Z$ with  $\delta(x) = x^i$ for all 
$x$ in a finite index subgroup of $\Zp^\times$, then $\mathscr{R}(\delta)$ is 
de Rham of Hodge--Tate weight $-i$ (as defined in \cite[\S2.2]{px}). We can now 
give the proof of Theorem \ref{parityapp}.

\begin{proof}[Proof of Theorem \ref{parityapp}]
	If $[F:\Q]$ is even, we let $D/F$ be the totally 
	definite quaternion algebra which is split at all finite places. If 
	$[F:\Q]$ is odd, we let $D/F$ be the totally definite quaternion algebra 
	which is split at all finite places except for the fixed place $v_0$.
	
	By the Jacquet--Langlands correspondence, we can find a compact open 
	subgroup $K^p \subset (D\otimes_F\A_{F,f}^p)^\times$, an integer $n \ge 
	1$, a finite order character $\epsilon = 
	\prod_{v|p}\epsilon_v: \prod_{v|p}T_{0,v} \rightarrow E_\lambda^\times$ and 
	a Hecke 
	eigenform $f \in 
	H^0(K^p\prod_{v|p}I_{v,1,n-1},\bigotimes_{v|p}\mc{L}(k_v,0,\epsilon_v))$ 
	with associated Galois representation
	$V_{f}$ isomorphic to $V_{g,\lambda}$. Moreover, if $g$ is $v$-(nearly) 
	ordinary at a 
	place $v|p$ we 
	choose $f$ so that its $U_v$-eigenvalue $\alpha_v$ satisfies 
	$v_p(\alpha_v) 
	= 0$. In particular, we have $v_p(\alpha_v) < k_v-1$ for all $v|p$. 
	
	Triviality of the central character of $g$ implies that each character 
	$\epsilon_v$ has trivial restriction to the diagonally embedded copy of 
	$\oo_v^\times$. The determinant $\det(V_f)$ is the cyclotomic character.
	
	Now we choose an ordering $v_1, v_2, \ldots, v_d$ of the places $v|p$. We 
	can now apply the results of sections \ref{ssec:halo} and \ref{ssec:galrep} 
	fixing the place $v = v_1$. By 
	Proposition \ref{prop:wtkpoints}, we deduce that we can find an integer 
	$n_1 \ge 
	n$, a finite extension $E_1/\Qp$, a finite order character $\epsilon_1 = 
	\prod_{v|p}\epsilon_{1,v}: \prod_{v|p}T_{0,v} \rightarrow E_1^\times$ and a 
	Hecke 
	eigenform $f_1 \in 
	H^0(K^p\prod_{v|p}I_{v,1,n_1-1},\mc{L}(2,0,\epsilon_{1,v_1})\bigotimes_{i=2}^d\mc{L}(k_{v_i},0,\epsilon_{1,v_i}))$
	 such that $f$ and $f_1$ give rise to classical points of a common 
	irreducible component $C$ of $\mc{E}(k^{v_1},0)$. Moreover, the 
	$U_v$-eigenvalues $\alpha_v$ of $f_1$ satisfy $v(\alpha_v) < k_v - 1$ for 
	all $v|p$. For $v_1$ this is part of the statement of Proposition 
	\ref{prop:wtkpoints}. For the other places, this follows from Lemma 
	\ref{lem:redlocus}. 
	
	Now we can apply Corollaries \ref{trivprime} and \ref{vtrifam}, together 
	with \cite[Thm.~A]{px} to deduce that the validity of the parity 
	conjecture 
	for $V_{g,\lambda}$ is equivalent to the validity of the parity conjecture 
	for 
	$V_{f_1}$. For the reader's convenience, we reproduce here the statement of  \cite[Thm.~A]{px}:
	
	\begin{theor}
		\label{T:thm A}
		Let $X$ be an irreducible reduced rigid analytic space over $\Qp$, and
		$\mathscr{T}$ a locally free coherent $\OO_X$-module equipped with a
		continuous, $\OO_X$-linear action of $G_{F,S}$, and a skew-symmetric
		isomorphism $j : \mathscr{T} \stackrel\sim\to \mathscr{T}^*(1)$.  Assume given,
		for each $v_i|p$, a short exact sequence
		\[
		\mathscr{S}_i : 0 \to \mathscr{D}_i^+ \to \mathbf{D}_\mathrm{rig}(\mathscr{T}|_{G_{F_{v_i}}}) \to
		\mathscr{D}_i^- \to 0
		\]
		of $(\varphi,\Gamma)$-modules over 
		$\mathscr{R}_{X}$, with
		$\mathscr{D}_i^+$ Lagrangian with respect to $j$.
		
		For a (closed) point $P \in X$, we put $V = \mathscr{T} \otimes_{\OO_X}
		\kappa(P)$ and $S_i = \mathscr{S}_i \otimes_{\OO_X} \kappa(P)$.  Let
		$X_\mathrm{alg}$ be the set of points $P \in X$ such that (1) for all 
		$i$, the sequence $S_i$ describes 
		$\mathbf{D}_\mathrm{rig}(V|_{G_{F_{v_i}}})$ as an extension of a de 
		Rham $(\varphi,\Gamma)$-module with non-negative Hodge--Tate weights by 
		a de Rham $(\varphi,\Gamma)$-module with negative Hodge--Tate 
		weghts\footnote{This is called the `Panchiskin condition' in 
		\cite{px}.}, 
		and (2) for all $w \in S \backslash
		S_\infty$ where $\mathscr{T}|_{G_{F_w}}$ is ramified, the associated
		Weil--Deligne representation $WD(V|_{G_{F_w}})$ is pure.
		
		Then the validity of the parity conjecture for $V$, relating the sign
		of its global $\varepsilon$-factor to the parity of the dimension of its
		Bloch--Kato Selmer group, is independent of $P \in X_\mathrm{alg}$.
	\end{theor}
Letting $U \subset \tilde{\mc{E}}$ be a Zariski open subset containing every 
noncritical classical point, as in the second part of Corollary \ref{vtrifam}, 
we apply the above theorem with $X = U \cap (\bigcap_{v'\ne v_1|p} 
\tilde{\mc{E}}^{v'-\mathrm{fs},\mathrm{good}}) $ (note that $X$ contains the 
points arising from both $f$ and $f_1$). The short exact sequences 
$\mathscr{S}_i$ are given by Corollary \ref{vtrifam} (for $i=1$) and Corollary 
\ref{trivprime} (for the other $i$). We can fix a skew-symmetric isomorphism 
$j$ (it exists since our Galois representations have cyclotomic 
determinant). The specialisation of $\mathscr{D}_i^+$ at a noncritical 
classical point of weight $(k_1,0)$ is de Rham of Hodge--Tate weight 
$-k_i/2$, and 
the specialisation of $\mathscr{D}_i^-$ is de Rham of Hodge--Tate weight 
$(k_i-2)/2$ (see 
Proposition 
\ref{strictri}). It follows from this and local-global compatibility at places 
in 
$S\backslash S_\infty$ that 
$X_{\mathrm{alg}}$ contains all the noncritical classical points of $X$ (in 
particular, the points arising from $f$ and $f_1$). Therefore the validity of 
the parity conjecture for 
$V_f$ is equivalent to the validity of the parity conjecture 
for $V_{f_1}$.
	
	We apply this argument $d-1$ more times, and  deduce that the 
	validity of the parity conjecture for $V_{g,\lambda}$ is equivalent to the 
	validity 
	of the parity conjecture for $V_{f_d}$, for a Hecke eigenform \[f_d \in 
	H^0(K^p\prod_{v|p}I_{v,1,n_d-1},\bigotimes_{i=1}^d\mc{L}(2,0,\epsilon_{d,v_i})).\]
	
	We moreover know that $f_d$ does not generate a 
	one-dimensional representation of $(D\otimes_F\A_{F})^\times$ (since it has 
	noncritical slope). Finally, by the 
	Jacquet--Langlands correspondence and \cite[Thm.~C]{nektame}, the parity 
	conjecture holds for $V_{f_d}$ (note that our Galois representations all 
	have cyclotomic determinant, so the automorphic representation associated 
	to $f_d$ has trivial central character), so the parity conjecture holds for 
	$V_f$.
\end{proof}

\section{The full eigenvariety}\label{sec:fullevar}
\subsection{Notation and preliminaries}
In this section we will consider an eigenvariety for automorphic forms on the 
quaternion algebra $D$ where the weights are allowed to vary at all places 
$v|p$. The set-up is very similar to \S\ref{ssec:partial} so we will be brief.

We set $T_0 = \prod_{v|p}T_{0,v}$, $I_{1} = \prod_{v|p} I_{v,1}$ and $I_{1,n} = 
\prod_{v|p}I_{v,1,n}$, and fix a 
compact open subgroup $K^p = \prod_{v\nmid p}K_v  \subset 
(\oo_D\otimes_F \A_{F,f}^{p})^\times$ such that $K = K^pI_1$ is neat.

 We let $X_p$ denote the kernel of  $N_{F/\Q}:\oo_{F,p}^\times \rightarrow 
 \Zp^\times$.

The weight space $\cW^\mathrm{full}$ is defined by letting 
\[\cW^{\mathrm{full}}(A) 
= \{\kappa \in 
\Hom_{cts}(T_{0},A^\times) : \kappa|_{Z(K)} = 1 \text{ and } \kappa|_{X_p} 
\text{ has finite order.}\}\] for $\Qp$-affinoid 
algebras $A$. We have $\dim(\cW^\mathrm{full}) = 1 + d$ and the locally 
algebraic 
points (defined as in the partial case) are Zariski dense in 
$\cW^\mathrm{full}$. If the $p$-adic Leopoldt conjecture holds for $F$ then 
$Z(K)$ has closure in 
$X_p$ a finite index subgroup, and the second condition in the definition of 
the points of weight space is automatic. 

For weights $\kappa \in \cW^\mathrm{full}(A)$ (together with a norm on $A$ 
which 
is adapted to $\kappa$) and $r \ge r_\kappa$ we get a Banach $A$-module of 
distributions $\mc{D}_\kappa^r$ equipped with a left action of the monoid 
$\prod_{v|p}\Delta_v$. We therefore obtain an $A$-module 
$H^0(K,\mc{D}_\kappa^r)$ equipped with an action of the Hecke algebra 
$\mathbb{T}$. 

We let $U_p = \prod_{v|p}U_v$, which is a compact operator on 
$H^0(K,\mc{D}_\kappa^r)$. We have its spectral variety $\mc{Z}^\mathrm{full} 
\rightarrow \cW^\mathrm{full}$ and a coherent sheaf $\mc{H}^\mathrm{full}$ on 
$\mc{Z}^\mathrm{full}$ coming from the modules $H^0(K,\mc{D}_\kappa^r)$, 
equipped 
with a map $\psi: \mathbb{T}\rightarrow \End(\mc{H}^\mathrm{full})$. We write 
$\mc{E}^\mathrm{full}$ for the eigenvariety associated to the eigenvariety 
datum 
$(\cW^\mathrm{full}\times \mb{A}^{1},\mc{H}^\mathrm{full},\mathbb{T},\psi)$ (or, equivalently, $(\mc{Z}^{full},\mc{H}^{full},\mb{T},\psi)$).
 $\mc{E}^\mathrm{full}$ is reduced and equidimensional of dimension $1+d$, with 
 a Zariski 
dense subset of classical points which we now describe.

Given integers $k=(k_v)_{v|p}$ and $w$ all of the same parity, 
and a finite order character $\epsilon: T_0 \rightarrow E^\times$ which is 
trivial on $\prod_{v|p}(1 + \varpi_v^n\oo_{F_v})^2$, we have an 
$I_{1,n-1}$-representation $\mc{L}(k,w,\epsilon) = 
\otimes_{v|p}\mc{L}_v(k_v,w,\epsilon)$.  As in \ref{ssec:lalgwts}, $k, w$ and 
$\epsilon$ correspond to a character $\kappa: T_{0} \rightarrow E^\times$, 
which we suppose is a point of $\cW^\mathrm{full}$ (in other words we suppose 
$\kappa|_{Z(K)} = 1$ and $\kappa|_{X_p}$ has finite order). We say that 
$\kappa$ is locally algebraic of weight 
$(k,w)$ and character $\epsilon$.

We then obtain a map (as in (\ref{maptoclass}))

\begin{equation*}\pi: 
H^0(K,\mathcal{D}^{r}_{\kappa}) 
\rightarrow  
H^0(K^pI_{1,n-1},\mc{L}(k,w,\epsilon)),\end{equation*}
where for each $v|p$ the action of $U_v$ on the target is the 
$\star$-\emph{action} defined by 
multiplying the standard action of $U_v$ by $p^{-(w-k_v+2)/2}$.

\begin{proposition}\label{class}
	For each $v|p$, let $h_v \in \Q_{\ge 0}$ with $h_v < k_v-1$. The map $\pi$ 
	induces an 
	isomorphism between the subspaces where $U_v$ acts with slope $\le h_v$ for 
	all $v|p$. 
\end{proposition}
\begin{proof}
As for Proposition \ref{partialclass}, this can be proved using the method of 
\cite[Thm.~3.2.5]{han}.
\end{proof}

We say that a point of $\mc{E}^\mathrm{full}$ is classical if its weight  is 
locally algebraic and the point corresponds to a Hecke eigenvector with 
non-zero image under the map $\pi$ above. We say that a classical point (with 
weight $(k,w)$) has noncritical slope if the corresponding $U_v$-eigenvalues 
have slope $< k_v-1$ for each $v|p$.

\begin{lemma}\label{etale}
	Let $z \in \mc{E}^\mathrm{full}$ be a classical point of weight $(k,w)$ and 
	character $\epsilon$ with noncritical slope. Let $\epsilon = 
	(\epsilon_1,\epsilon_2)$, where $\epsilon_i$ is a character of 
	$\oo_{F,p}^\times$. If $\epsilon_1|_{\oo_{F_v}^\times} = 
	\epsilon_2|_{\oo_{F_v}^\times}$ suppose moreover that $v_p(U_v(z)) \ne 
	\frac{k_v-1}{2}$. Then $\mc{E}^\mathrm{full}$ is \'{e}tale over 
	$\cW^\mathrm{full}$ at 
	$z$.
\end{lemma}
\begin{proof}
This can be proved as in
\cite[Thms.~4.8,~4.10]{chefern}. Under our 
assumptions, the Hecke algebra 
$\mathbb{T}$ acts semisimply on the space of classical automorphic forms of 
fixed weight and character. We denote the residue field $k(\kappa(z))$ by $L$. 
Replacing $z$ by a point (which we also call $z$) lying over it in 
$\mc{E}^\mathrm{full}_L$, it suffices to show that $\mc{E}^\mathrm{full}_L$ is 
\'{e}tale over $\mc{W}^\mathrm{full}_L$ at $z$. By the construction of the 
eigenvariety, we can 
suppose that we have a geometrically connected (smooth) $L$-affinoid 
neighbourhood 
$B$ of the 
weight $\kappa(z) \in \mc{W}^\mathrm{full}_L$ and a finite locally free 
$\oo(B)$-module $M$, equipped with an $\oo(B)$-linear $\mathbb{T}$-action such 
that:\begin{enumerate}
	\item The affinoid spectrum $V$ of the image of $\oo(B)\otimes\mathbb{T}$ 
	in $\End_{\oo(B)}(M)$ is an open neighbourhood of $z \in 
	\mc{E}^\mathrm{full}$.
	\item For each point $x \in B$ and algebraic closure 
	$\overline{k(x)}$ of the residue field $k(x)$ there is an isomorphism 
	of $\overline{k(x)}\otimes\mathbb{T}$-modules 
	\[M\otimes_{\oo(B)}\overline{k(x)} \cong 
	\bigoplus_{y \in \kappa^{-1}(x)\cap 
	V}\bigoplus_{\iota \in \Hom_{k(x)}(k(y),\overline{k(x)})} 
	H^0(K,\mc{D}_x^r)\otimes_{k(x)}\overline{k(x)}[\iota\circ\psi_y]\] where 
	$\psi_y$ denotes the ($k(y)$-valued) system of Hecke 
	eigenvalues associated to $y$, and $[\iota\circ\psi_y]$ denotes the 
	generalised eigenspace.
	\item $\kappa^{-1}(\kappa(z)) = \{z\}$ and the natural $L$-algebra 
	surjection $\oo(V) \rightarrow k(z)$ has a section.
\end{enumerate}

Now we apply the classicality criterion (Proposition \ref{class}), strong 
multiplicity one, and \cite[Lem.~4.7]{chefern} as in 
\cite[Thm.~4.10]{chefern} to deduce that there is a Zariski-dense subset $Z_0 
\subset B$ of locally algebraic weights with residue field $L$ such that 
$\kappa$ is \'{e}tale at each 
point of $\kappa^{-1}(Z_0)$, and for each $\chi \in Z_0$, 
$\kappa^{-1}(\chi)\cap V$ consists of a single point with residue field $k(z)$.

We can now show that the map $k(z)\otimes_L \oo(B) \rightarrow \oo(V)$ obtained 
from 
assumption (3) is an isomorphism. Since $k(z)\otimes_L \oo(B)$ is a normal 
integral domain, it suffices to show that the map $\pi: \Spec(\oo(V)) 
\rightarrow 
\Spec(k(z)\otimes_L \oo(B))$ is birational. By generic flatness, there is a 
dense open subscheme $U$ of the target such that $\pi|_U$ is finite flat. Since 
$\pi$ maps irreducible components surjectively onto irreducible components, 
$\pi^{-1}(U)$ is a dense open subscheme of the source. Since $Z_0 \cap U$ is 
non-empty, $\pi|_U$ has degree one and is therefore an isomorphism, which shows 
that $\pi$ is birational as desired. In fact, it is not necessary to use the 
normality of $B$ (as explained to us by Chenevier) --- $\oo(V)$ is a subalgebra 
of $\End_{k(z)\otimes_L\oo(B)}(k(z)\otimes_L M)$ and each element of $\oo(V)$ 
acts as a scalar in $k(z)$ when we specialise at a weight $\chi \in Z_0$. Since 
$B$ is reduced and $Z_0$ is Zariski dense in $B$, this is enough to conclude 
that each element of $\oo(V)$ acts as a scalar (in $k(z)\otimes_L\oo(B)$) on 
$k(z)\otimes_L M$.
\end{proof}

\subsection{Mapping from partial eigenvarieties to the full eigenvariety}
Now we fix a place $v|p$ and consider the eigenvariety $\mc{E}(k^v,w)$ for 
fixed $k^v, w$ as in 
section \ref{ssec:halo}. For this partial eigenvariety, we fix the level 
structure 
$K'_{v'} = I'_{v',1,n-1}$ for 
each place $v'|p$ with $v' \ne v$, and let $K' = K^pI_{v,1}\prod_{v'\ne 
v}K'_{v'}$. We denote by $\mc{E}(k^v,w)^{U_p-fs} \subset 
\mc{E}(k^v,w)$ the union of 
irreducible components (with the reduced subspace structure) given by the image 
of $\cap_{v\neq v'|p}\widetilde{\mc{E}}^{v'-fs}$ (see Definition 
\ref{vprimefs}).

We can also construct a spectral variety $\mc{Z}^{U_p}(k^v,w)$ and eigenvariety 
$\mc{E}^{U_p}(k^v,w)$ using the compact operator $U_p$ instead of $U_v$. The 
closed points of $\mc{E}^{U_p}(k^v,w)$ and $\mc{E}(k^v,w)^{U_p-fs}$ correspond 
to the same systems of Hecke eigenvalues, so the following lemma should be no 
surprise.

\begin{lemma}\label{switchspectral}
There is a canonical isomorphism 
$\mc{E}^{U_p}(k^v,w)^{red}\cong\mc{E}(k^v,w)^{U_p-fs}$, 
over $\mc{W}$, compatible with the Hecke operators. 
\end{lemma}
\begin{proof}
	Suppose we have a slope datum $(U,h)$ (see \cite[Defn.~2.3.1]{extended}) 
	corresponding to an open subset $\mc{Z}^{U_p}_{U,h} \subset 
	\mc{Z}^{U_p}(k^v,w)$. We therefore have a slope decomposition 
	\[H^0(K',\mc{L}^v(k^v,w)\otimes_{\Zp}\mathcal{D}^{r}_{U}) = M = M_{\le h} 
	\oplus M_{> h}\] and a corresponding factorisation $\det(1-U_vT) = 
	QS$ of the characteristic power series of $U_v$ on $M$, so we 
	obtain a closed immersion $Z(Q) := \{Q = 0\} \hookrightarrow 
	\mc{Z}^{U_v}(k^v,w)|_U$. The module $M_{\le h}$ defines a coherent sheaf on 
	$Z(Q)$. 
	
	Gluing, we obtain a rigid space $\mc{Z}^{U_p,U_v}$ 
	equipped with a closed immersion $\mc{Z}^{U_p,U_v} \hookrightarrow 
	\mc{Z}^{U_v}(k^v,w)$, together with a coherent sheaf $\mc{H}^{U_p,U_v}$ on 
	$\mc{Z}^{U_p,U_v}$. 
	
	Over $\mc{Z}^{U_p}_{U,h}$, $\mc{E}^{U_p}(k^v,w)$ is given by 
	the spectrum 
	of the affinoid algebra 
	\[\mc{T}_{U,h} = \mathrm{im}(\mathbb{T}\otimes_{\Zp}\oo(U) 
	\rightarrow \End_{\oo(U)}M_{\le h}).\] The eigenvariety 
	associated to the 
	datum $(\mc{W}\times \mb{A}^{1},\mc{H}^{U_p,U_v},\mathbb{T},\psi)$ has the 
	same description (since $U_v \in \mathbb{T}$), so these two eigenvarieties 
	are canonically isomorphic.
	
	On the other hand, we can identify the closed points of 
	$\mc{E}(k^v,w)^{U_p-fs}$ 
	and $\mc{E}^{U_p}(k^v,w)$, since they correspond to the same systems of 
	Hecke eigenvalues. It follows from \cite[Thm.~3.2.1]{irreducible}, 
	applied to the eigenvariety data 
	$(\mc{W}\times \mb{A}^{1},\mc{H}^{U_p,U_v},\mathbb{T},\psi)$ and 
	$(\cW \times \mb{A}^{1},\mc{H},\T,\psi)$
	  that there is a canonical closed immersion $\mc{E}^{U_p}(k^v,w)^{red} 
	  \hookrightarrow \mc{E}(k^v,w)$, which induces the desired isomorphism 
	  $\mc{E}^{U_p}(k^v,w)^{red}\cong \mc{E}(k^v,w)^{U_p-fs}$.
\end{proof}

To compare $\mc{E}(k^v,w)^{U_p-fs}$ with $\mc{E}^\mathrm{full}$, we need to fix 
a 
character $\epsilon^v: \prod_{v' \ne v}I_{v',1,n-1}/I'_{v',1,n-1} \rightarrow 
\Qp^\times$, and consider the eigenvarieties  $\mc{E}^{U_p}(k^v,w,\epsilon^v)$, 
$\mc{E}(k^v,w,\epsilon^v)$ given by replacing the coherent sheaf $\mc{H}$ in 
the eigenvariety data with the $\epsilon^v$-isotypic piece 
$\mc{H}(\epsilon^v)$. These are unions of connected components in the original 
eigenvarieties. They are supported over the union of connected components 
$\mc{W}(\epsilon^v) 
\subset \mc{W}$ given by weights $\kappa_v$ such that 
$\kappa_v\epsilon^v\chi^v|_{Z(K)} = 1$, where $\chi^v$ is the highest weight of 
$\mc{L}^v(k^v,w)$. Note that in general points $\kappa_v \in \mc{W}$ satisfy 
the weaker condition 
that $\kappa_v\chi^v|_{Z(K')} = 1$.

\begin{definition}We write $\iota_{(k^v,w,\epsilon^v)}: \mc{W}(\epsilon^v) 
\hookrightarrow 
\mc{W}^\mathrm{full}$ for the closed immersion defined by 
\[\iota_{(k^v,w,\epsilon^v)}(\kappa_v)= \kappa_v\epsilon^v\chi^v\] 
where 
$\chi^v$ is the highest weight of $\mc{L}^v(k^v,w)$.
\end{definition}

\begin{lemma}\label{embedpartial}
	There is a canonical closed immersion \[\mc{E}(k^v,w,\epsilon^v)^{U_p-fs} 
	\hookrightarrow \mc{E}^\mathrm{full},\] lying over the map 
	$\iota_{(k^v,w,\epsilon^v)}: 
	\mc{W}(\epsilon^v) \hookrightarrow \mc{W}^\mathrm{full}$, compatible with 
	the Hecke 
	operators. 
\end{lemma}
\begin{proof}
	Applying Lemma \ref{switchspectral}, we replace 
	$\mc{E}(k^v,w,\epsilon^v)^{U_p-fs}$ in the statement with 
	$\mc{E}^{U_p}(k^v,w,\epsilon^v)^{red}$. The classical points of 
	$\mc{E}^{U_p}(k^v,w,\epsilon^v)^{red}$ naturally correspond to a subset of 
	the classical points of $\mc{E}^\mathrm{full}$, and we apply 
	\cite[Thm.~3.2.1]{irreducible} to obtain the desired closed immersion.
\end{proof}

\begin{corollary}
	Let $C$ be an irreducible component of the full eigenvariety 
	$\mc{E}^\mathrm{full}$. Then 
	$C$ contains a noncritical slope classical point of weight 
	$(2,\ldots,2,0)$\footnote{or any other algebraic weight}.
\end{corollary}
\begin{proof}
First we note that $C$ contains a noncritical slope classical point $x_0$ of 
weight 
$(k_1,\ldots,k_d,w)$ and character $\epsilon = (\epsilon_1,\epsilon_2)$ with 
$w$ even (since the image of $C$ in weight space is 
Zariski open). We can moreover assume that the characters 
$\epsilon_1|_{\oo_{F_v}^\times}$ and $\epsilon_2|_{\oo_{F_v}^\times}$ are 
distinct for all $v|p$ (for example by considering weights in a sufficiently 
small `boundary poly-annulus' of the weight space). We do this to ensure tht 
$x_0$ will be \'{e}tale over weight space. As in the 
proof of Theorem \ref{parityapp}, we obtain 
a sequence 
$x_1, x_2, \ldots, x_d$ of noncritical slope classical points where each $x_i$ 
has a weight $k(i)$ with $k(i)_1 = k(i)_2 = \cdots k(i)_i =2$, with the 
remaining weight components the same as $x_0$. The $v_{i+1}$ part of the 
character $\epsilon$ changes when we move from $x_i$ to $x_{i+1}$, but it keeps 
the property that it has distinct components. So, by Lemma \ref{etale}, the map 
from $\mc{E}^\mathrm{full}$ to $\mc{W}^\mathrm{full}$ is 
\'{e}tale at each of the points $x_i$

For each $i$, $x_i$ and 
$x_{i+1}$ lie in the image of an irreducible component of 
$\mc{E}(k(i)^{v_{i+1}},w,\epsilon^{v_{i+1}})^{U_p-fs}$ under the closed 
immersion of 
Lemma \ref{embedpartial} (for some choice of $\epsilon^{v_{i+1}}$). It follows 
that 
$x_i$ and $x_{i+1}$ lie in a common irreducible component of 
$\mc{E}^\mathrm{full}$.

Since the map from $\mc{E}^\mathrm{full}$ to $\mc{W}^\mathrm{full}$ is 
\'{e}tale at each of the points $x_i$, $x_i$ is a smooth point of 
$\mc{E}^\mathrm{full}$ and is therefore contained in a unique irreducible 
component of $\mc{E}^\mathrm{full}$. We deduce that $x_i \in C$ for 
all $i$. The point $x_d$ has weight $(2,\ldots,2,w)$. Finally, there 
is a one-dimensional family of twists by powers of the cyclotomic character 
connecting $x_d$ to a noncritical slope classical point of weight 
$(2,\ldots,2,0)$, 
which gives the desired point of $C$.
\end{proof}
\begin{remark}
	At least for irreducible components $C$ whose associated mod $p$ Galois 
	representation $\rhobar$ is absolutely irreducible, one can deduce cases of 
	the parity conjecture directly from the above corollary, using 
	\cite[Thm.~A]{px}. Without this assumption on $\rhobar$, the Galois 
	pseudocharacter over $C$ cannot be automatically lifted to a genuine family 
	of Galois representations over $C$. This is why our proof of Theorem 
	\ref{parityapp} instead uses one-dimensional families coming from the 
	partial 
	eigenvarieties, where we can apply Proposition \ref{liftdet}.
\end{remark}

\bibliographystyle{alpha}
\bibliography{../halo}

\end{document}